\documentclass[11pt,reqno]{amsart}
\newtheorem{theorem}{Theorem}

\newtheorem{definition}{Definition}

\newtheorem{lemma}{Lemma}
\newtheorem{proposition}{Proposition}
\newtheorem{remark}{Remark}
\newcommand{\new}[1]{#1}
\newcommand{\old}[1]{}
\numberwithin{equation}{section}
\usepackage{amsmath}
\usepackage{amssymb}
\usepackage{amsfonts}
\usepackage{graphicx}
\usepackage[numbers]{natbib}
\usepackage{subfigure}
\usepackage{caption}
\usepackage{import}
\usepackage{xifthen}
\usepackage{pdfpages}
\usepackage{transparent}

\usepackage{comment}
\usepackage{tikz}
\usepackage{url}
\usepackage[linktocpage]{hyperref}
\hypersetup{
    colorlinks=true,
    citecolor=blue,
    linkcolor=blue,
    urlcolor=blue,
}
\calclayout
\begin{document}

\title[Lie--Poisson discretization for incompressible MHD]{Spatio-temporal Lie--Poisson discretization for incompressible magnetohydrodynamics on the sphere}

\author{Klas Modin}
\address{Klas Modin: Department of Mathematical Sciences, Chalmers University of Technology and University of Gothenburg, 412 96 Gothenburg, Sweden}
\email{klas.modin@chalmers.se}

\author{Michael Roop}
\address{Michael Roop: Department of Mathematical Sciences, Chalmers University of Technology and University of Gothenburg, 412 96 Gothenburg, Sweden}
\email{michael.roop@chalmers.se}

\subjclass[2020]{37M15; 65P10; 53D20; 76W05}

\keywords{magnetohydrodynamics, Lie--Poisson structure, magnetic extension, Casimirs, Hamiltonian dynamics, symplectic Runge-Kutta integrators}

\begin{abstract}
We give a structure preserving spatio-temporal discretization for incompressible magnetohydrodynamics (MHD) on the sphere. Discretization in space is based on the theory of geometric quantization, which yields a spatially discretized analogue of the MHD equations as a finite-dimensional Lie--Poisson system on the dual of the magnetic extension Lie algebra $\mathfrak{f}=\mathfrak{su}(N)\ltimes\mathfrak{su}(N)^{*}$. We also give accompanying structure preserving time discretizations for Lie--Poisson systems on the dual of semi-direct product Lie algebras of the form $\mathfrak{f}=\mathfrak{g}\ltimes\mathfrak{g^{*}}$, where $\mathfrak{g}$ is a $J$-quadratic Lie algebra. The time integration method is free of computationally costly matrix exponentials. We prove that the full method preserves \new{a modified Lie--Poisson structure and corresponding Casimir functions, and that the modified structure and} Casimirs converge to the continuous ones. The method is demonstrated for two models of magnetic fluids: incompressible magnetohydrodynamics and Hazeltine's model.
\end{abstract}

\maketitle
\tableofcontents
\vspace{-0.5cm}
{\bf Acknowledgement.} This work was supported by the Knut and Alice Wallenberg Foundation, grant number WAF2019.0201, and by the Swedish Research Council, grant number 2022-03453. The work of M.R. is supported by the Royal Swedish Academy of Sciences, grant numbers MG2023-0001, MA2024-0034, MG2024-0050. The authors would like to thank Darryl Holm and Philip Morrison for pointing us to Hazeltine's model for magnetic fluids.
\section{Introduction}
The equations of incompressible magnetohydrodynamics (MHD) describe the evolution of the velocity $v(t,x)$ of an ideal charged fluid and its magnetic field $B(t,x)$ on a two- or three-dimensional Riemannian manifold $M$:
\begin{equation}
\label{MHD}
\left\{
\begin{aligned}
&\dot v+\nabla_{v}v=-\nabla p+\mathrm{curl}B\times B, \\
&\dot B=L_{v}B,\\
&\operatorname{div}B=0,\\
&\operatorname{div}v=0.
\end{aligned}
\right.
\end{equation}
Here, $p(t,x)$ is a pressure function, $L_{v}$ denotes the Lie derivative along the vector field $v(t,x)$, and $\nabla_{v}v$ is the covariant derivative of the vector field $v$ along itself.

The MHD system \eqref{MHD} admits a Hamiltonian formulation in terms of a \emph{Lie--Poisson structure} \citep{Arn,ArnKh,VD,MoGr1980} on the dual of the semi-direct product Lie algebra
\begin{equation*}
\mathfrak{imh}=\mathfrak{X}_{\mu}(M)\ltimes\mathfrak{X}_{\mu}^{*}(M).
\end{equation*}
The Hamiltonian on $\mathfrak{imh}^*$ is given by 
\begin{equation*}
H(v,B)=\frac{1}{2}\int\limits_{M}\left(|v|^{2}+|B|^{2}\right)\mu .
\end{equation*}
Here, $\mathfrak{X}_{\mu}(M)$ is the Lie algebra of divergence-free vector fields, whereas $\mathfrak{X}_{\mu}^{*}(M)$ and $\mathfrak{imh}^*$ denote the (smooth) dual spaces.
Physically, the Hamiltonian represents the energy of the magnetic fluid.

Geometrically, system \eqref{MHD} is the Poisson reduction of a canonical Hamiltonian system on the cotangent bundle
\begin{equation*}
T^{*}\left(\mathrm{Diff}_{\mu}(M)\ltimes \mathfrak{X}_{\mu}^{*}(M)\right), 
\end{equation*}
where the subscript $\mu$ stands for the Riemannian volume form, and $\mathrm{Diff}_{\mu}(M)$ is the group of volume-preserving diffeomorphisms of $M$, i.e., diffeomorphisms of $M$ that leave the differential form $\mu$ invariant:
\begin{equation*}
\mathrm{Diff}_{\mu}(M)=\left\{\varphi\in\mathrm{Diff}(M)\mid\varphi^{*}\mu=\mu\right\}.
\end{equation*}

The Hamiltonian nature of the flow \eqref{MHD} implies a multitude of conservation laws. 
\new{In particular, recall that a function on $\mathfrak{imh}^*$ that Poisson commutes with every other function on $\mathfrak{imh}^*$ is called a Casimir function.}
In 3-D, they are \textit{magnetic helicity} and \textit{cross-helicity}; Khesin, Peralta-Salas, and Yang~\cite{KhPerYang2019} showed that these are the only independent Casimirs. 
In 2-D, there are an infinite number of Casimirs (detailed below). 
These conservation laws, and the underlying Lie--Poisson structure, significantly restrict which states are possible to reach from a given initial state.
They thereby influence the long time qualitative behaviour in phase space.
Indeed, to capture the qualitative behaviour in long time numerical simulations, one should use discretizations that preserve the rich geometric structure in phase space (for a detailed motivation of structure preserving schemes in the case of plasma physics, see the review paper by Morrison~\cite{Mo2017a}).

But infinite-dimensional Lie--Poisson structures, such as $\mathfrak{imh}^*$ for MHD, are strikingly rigorous; all traditional spatial discretizations, including all finite element methods based on discrete exterior calculus, fail to admit a finite-dimensional Lie--Poisson formulation.
In addition, Lie--Poisson preserving time discretizations (integrators) are hard to come by. 
Nevertheless, to find such structure preserving discretizations is critical for qualitatively reliable long time simulations, motivated by the intensified study of stellarators, where 2-D MHD provide a simple model for low beta tokamak dynamics \cite{KrTassGrass2016}.

The goal of this paper is to develop, for the sphere $\mathrm{S}^2$, a spatio-temporal discretization of the MHD system \eqref{MHD} that preserves \new{a modified, finite-dimensional} Lie--Poisson structure, including the corresponding Casimir conservation laws. 
To this end, we draw on two bodies of previous work.
First, that of Zeitlin~\cite{Ze1991,Zeit,Ze2005}, who used quantization theory to derive a Lie--Poisson preserving spatial discretization for the incompressible Euler equations on the flat torus and later extended it to MHD.
Second, that of Modin and Viviani~\cite{ModViv,ModViv1}, who for the spherical domain $\mathrm{S}^2$ developed a tailored Lie--Poisson preserving temporal discretization \new{used for numerical studies of the long time behaviour}, and together with Cifani~\cite{CiViMo2023} addressed computational efficiency.

Let us give a brief overview of the approach. For spatial discretization, the main tool is the theory of \textit{Berezin--Toeplitz quantization} \cite{Hopp,Hopp1,Bord1991,Bord1994,HoppYau}. 
The basic idea is to replace the infinite-dimensional Poisson algebra of smooth functions by the finite-dimensional Lie algebra $\mathfrak{su}(N)$ of skew-hermitian, trace free matrices. 
Together with a quantized Laplacian on $\mathfrak{su}(N)$, one then obtains a finite-dimensional approximation of Euler's equations as a matrix flow --- \emph{Zeitlin's model}. The present paper extends this approach to models describing the motion of incompressible magnetized fluids on $\mathrm{S}^2$. Indeed, the spatially discretized analogue of the MHD equations constitutes a Lie--Poisson flow on the dual of the Lie algebra $\mathfrak{f}=\mathfrak{su}(N)\ltimes\mathfrak{su}(N)^{*}$, usually referred to as the \textit{magnetic extension} of $\mathfrak{su}(N)$.
Next, for temporal discretization, it is natural to consider \emph{isospectral symplectic Runge-Kutta integrators} (IsoSRK) \cite{ModViv}.
These schemes yield Lie--Poisson integrators for any \emph{reductive and $J$-quadratic} Lie algebra $\mathfrak{g}$ (see details below), which include all the classical Lie algebras.
However, the magnetic extension $\mathfrak{g}\ltimes \mathfrak{g}^*$ is \emph{not} reductive and \emph{not} defined by a $J$-quadratic condition; we need an extension of IsoSRK.
Such an extension is developed in this paper.

Although MHD on $\mathrm{S}^2$ is our main concern, the semi-direct product approach covers a large variety of dynamical systems arising in mathematical physics \cite{HMR,ThMo2000}. 
Among them are:
\begin{itemize}
  \item the Kirchhoff equations \cite{ArnKh,KhMisMod,Kir,VD}, describing a rigid body moving in an ideal fluid, as a Lie--Poisson system on the dual of $\mathfrak{e}(3)=\mathfrak{so}(3)\ltimes\mathbb{R}^{3}$;
  \item the barotropic Euler equations describing the motion of a compressible fluid \cite{KhMisMod,MRW}, as a Lie--Poisson system on the dual of $\mathfrak{s}=\mathfrak{X}(M)\ltimes C^{\infty}(M)$;
  \item Hazeltine's equations describing magnetized plasma \cite{Holm,Haz,HazHolm}. 
\end{itemize}
2-D MHD together with these and other examples underline the need for structure preserving numerical methods for Lie--Poisson systems of semi-direct product Lie algebras. 

The paper is organized as follows.
In section~\ref{sec:mhd_vorticity} we give the vorticity formulation for the 2-D MHD equations.
This formulation is the basis for the spatial discretization given in section~\ref{sec:spatial_discretization}, together with convergence results for the discretized Casimirs.
Temporal Lie--Poisson discretizations are then derived and analysed in section~\ref{sec:temporal_discretization}. 
The framework is extended to Hazeltine's equations in section~\ref{sec:hazeltine}, and to the Kirchhoff equations in section~\ref{sec:kirchhoff_equations}.
Numerical examples are given in section~\ref{sec:numerical_simulations}. 

\section{Vorticity formulation for MHD equations}\label{sec:mhd_vorticity}
In this section, we work on two-dimensional Riemannian manifolds $(M,g)$ without boundary and with trivial first co-homology (i.e., no ``holes'').
First, since the vector fields $B(t,x)$ and $v(t,x)$ are divergence-free and the co-homology is trivial, one can introduce two smooth functions $\theta\in C^{\infty}(M)$ and $\psi\in C^{\infty}(M)$ corresponding to the Hamiltonians for the vector fields $v$ and $B$:
\begin{equation*}
v=X_{\psi},\quad B=X_{\theta}.
\end{equation*}
The function $\psi$ is called the \textit{stream function}.
Similarly, we refer to $\theta$ as the \textit{magnetic stream function}.

Next, we define the \textit{vorticity function} $\omega\in C^{\infty}(M)$ and the \emph{magnetic vorticity} $\beta\in C^\infty(M)$ by
\begin{equation*}
  \beta = \Delta\theta, \qquad \omega = \Delta\psi.
\end{equation*}
\begin{proposition}[\cite{VD,MoGr1980}]
The vorticity formulation of the 2-D MHD equations \eqref{MHD} is
\begin{equation}
\label{MHDvort}
\left\{
\begin{aligned}
&\dot\omega=\left\{\omega,\psi\right\}+\left\{\theta,\beta\right\}, \quad &\omega =\Delta\psi, \\
&\dot\theta=\left\{\theta,\psi\right\}, \quad &\beta=\Delta\theta, 
\end{aligned}
\right.
\end{equation}
where $\left\{\cdot,\cdot\right\}$ is the Poisson bracket on $M$.
\end{proposition}

System \eqref{MHDvort} admits a Hamiltonian formulation in terms of a non-canonical Poisson bracket, see \cite{MoGr1980}. The corresponding Hamiltonian is
\begin{equation*}
H=\frac{1}{2}\int\limits_{M}\left(\omega\Delta^{-1}\omega+\theta\Delta\theta\right)\mu.
\end{equation*}
The Casimir invariants for system \eqref{MHDvort} are
\begin{equation}
\label{enstr}
\mathcal{C}_{f}=\int\limits_{M}f(\theta)\mu,\quad I_{g}=\int\limits_{M}\omega g(\theta)\mu,
\end{equation}
for any choice of smooth functions $f\colon\mathbb{R}\to\mathbb{R}$ and $g\colon\mathbb{R}\to\mathbb{R}$. The function $I_{g}$ is the two-dimensional analogue of the cross-helicity Casimir.
\begin{remark}
The Casimir $I_{g}$ in \eqref{enstr} is a more general invariant compared to the conventional definition of cross-helicity. 
Indeed, the Casimir $I_{g}$ corresponds to conventional cross-helicity for $g(x)=x$. 
We shall, however, refer to $I_{g}$ as \textit{cross-helicity} even for a general function $g$.
\end{remark}

\begin{remark}
  Due to Stokes' theorem, the vorticity functions $\omega$ and $\beta$ have zero mean
  \begin{equation*}
  \int\limits_{M}\omega\mu=\int\limits_{M}\beta\mu=0,
  \end{equation*}
  which reflects zero circulation.
  Since Hamiltonian functions are defined up to a constant, it is no restriction to assume that also
  \begin{equation*}
  \int\limits_{M}\psi\mu=\int\limits_{M}\theta\mu=0,
  \end{equation*}
  and therefore system \eqref{MHDvort} evolves on the space of pairs of zero-mean functions $C_{0}^{\infty}(M)$.
\end{remark}

\section{Spatial discretization of MHD equations}\label{sec:spatial_discretization}
In this section, we present a spatial discretization of the incompressible MHD equations on the sphere $\mathrm{S}^2$ based on the theory of quantization \cite{Hopp1,Bord1991,Bord1994}. 
In contrast to standard discretization schemes for systems of PDEs, such as finite element methods, we focus on conservation of the underlying geometric structure in phase space and the corresponding Casimirs \eqref{enstr}. 
Namely, we replace the infinite-dimensional Poisson algebra $(C_{0}^{\infty},\left\{\cdot,\cdot\right\})$ with a finite-dimensional analogue: skew-Hermitian matrices with zero trace $(\mathfrak{su}(N),[\cdot,\cdot])$. 
The sequence of Lie algebras $(\mathfrak{su}(N),[\cdot,\cdot])$ converges (in a weak sense) to the Lie algebra $(C_{0}^{\infty},\left\{\cdot,\cdot\right\})$ as $N\to\infty$, as we shall briefly review next. 
Thereafter, the spatially discretized analogue of \eqref{MHDvort} is a Lie--Poisson flow on the dual $\mathfrak{f}^{*}$ of the semi-direct product Lie algebra $\mathfrak{f}=\mathfrak{su}(N)\ltimes\mathfrak{su}(N)^{*}$.
\subsection{Quantization on the sphere}
We start with the definition of an $\mathfrak{L}_{N}$-quasilimit \cite{Bord1991,Bord1994}.

Let $(\mathfrak{L},[\cdot,\cdot])$ be a complex (real) Lie algebra and let $(\mathfrak{L}_{N},[\cdot,\cdot]_{N})$ be an indexed sequence of complex (real) Lie algebras with $N\in\mathbb{N}$, equipped with metrics $d_{N}$ and a family of linear maps $p_{N}\colon\mathfrak{L}\to\mathfrak{L}_{N}$.
\begin{definition}
The Lie algebra $(\mathfrak{L},[\cdot,\cdot])$ is said to be an $\mathfrak{L}_{N}$-\textit{quasilimit}, if
\begin{itemize}
\item all $p_{N}$ are surjective for $N\gg0$,
\item if for all $x,y\in\mathfrak{L}$ we have $d_{N}(p_{N}(x),p_{N}(y))\to0$, as $N\to\infty$, then $x=y$,
\item for all $x,y\in\mathfrak{L}$ we have $d_{N}(p_{N}([x,y]),[p_{N}(x),p_{N}(y)]_{N})\to0$, as $N\to\infty$.
\end{itemize}
\end{definition}
Now we explicitly specify $M$ to be the two-dimensional sphere $\mathrm{S}^{2}$, which is a symplectic manifold with symplectic form $\Omega$ given by the area form.
The associated Poisson bracket on $\mathrm{S}^{2}$ is given by
\new{
\begin{equation}
\label{PoiBr}
\left\{f,g\right\}=\Omega(X_{f},X_{g}),\quad f,g\in C_{0}^{\infty}(\mathrm{S}^{2}).
\end{equation}
}
Equipped with the bracket \eqref{PoiBr}, the set $C_{0}^{\infty}(\mathrm{S}^{2})$ becomes an infinite-dimensional Poisson algebra $(C_{0}^{\infty}(\mathrm{S}^{2}),\left\{\cdot,\cdot\right\})$ with an orthogonal basis (with respect to $L^2$) given by spherical harmonics $Y_{lm}(\vartheta,\phi)$:
\begin{equation*}
Y_{lm}(\vartheta,\phi)=\sqrt{\frac{2l+1}{4\pi}\frac{(l-m)!}{(l+m)!}}\mathcal{P}_{l}^{m}(\cos\vartheta)e^{\mathrm{i}m\phi},\quad l\ge1,\,m=-l,-l+1,\ldots,l,
\end{equation*}
where $\mathcal{P}_{l}^{m}$ are the associated Legendre functions. 
Then, elements of the Poisson algebra $(C_{0}^{\infty}(\mathrm{S}^{2}),\left\{\cdot,\cdot\right\})$ are approximated by matrices in the following way \cite{Hopp,Hopp1}. 
An approximating sequence is given by the matrix Lie algebras $(\mathfrak{su}(N),[\cdot,\cdot]_{N})$, where $[\cdot,\cdot]_{N}=\frac{1}{\hbar}[\cdot,\cdot]$ for $\hbar = 2/\sqrt{N^2-1}$ is a rescaling of the matrix commutator $[\cdot,\cdot]$. 
The family of projections 
\begin{equation}
\label{proj1}
p_{N}\colon C_{0}^{\infty}(\mathrm{S}^2)\to\mathfrak{su}(N),\quad p_{N}\colon Y_{lm}\mapsto \mathrm{i}T^{N}_{lm}
\end{equation}
is defined as follows for the basis element $Y_{lm}$:
\begin{equation}
\label{proj2}
\left(T^{N}_{lm}\right)_{m_{1}m_{2}}=(-1)^{[(N-1)/2]-m_{1}}\sqrt{2l+1}
\begin{pmatrix}
\frac{N-1}{2}&l&\frac{N-1}{2}\\
-m_{1} & m &m_{2}
\end{pmatrix},
\end{equation}
where $(:::)$ denotes the Wigner 3j-symbol. 
Then, the following result of $\mathfrak{L}_{N}$-convergence holds:

\begin{theorem}[\cite{Bord1991,Bord1994}]\label{thm:Lalpha}
For any choice of matrix norms $d_{N}$, the sequence of finite-dimensional Lie algebras $(\mathfrak{su}(N),[\cdot,\cdot]_{N})$, $N\in\mathbb{N}$, with projections defined by \eqref{proj1}-\eqref{proj2}, is an $\mathfrak{L}_{N}$-approximation of the infinite-dimensional Poisson algebra $(C_{0}^{\infty}(\mathrm{S}^2),\left\{\cdot,\cdot\right\})$ with the Poisson bracket \eqref{PoiBr}.
\end{theorem}

Let us introduce the matrix operator norm (also called the spectral norm):
\begin{equation*}
\|A\|_{L_{N}^{\infty}}=\sup\limits_{\|x\|=1}\|Ax\|,\quad A\in\mathfrak{u}(N),
\end{equation*}
where $\|\cdot\|$ is the Euclidean norm.

The following results give the convergence rate for the $\mathfrak{L}_{N}$-approximation in Theorem~\ref{thm:Lalpha}.
\begin{theorem}[\cite{Bord1994}]
For every $f,g\in C^{\infty}(\mathrm{S}^{2})$ there exists $c>0$, such that
\begin{eqnarray*}
&\displaystyle\|f\|_{L^{\infty}}-c\hbar\le\|p_{N}(f)\|_{L_{N}^{\infty}}\le\|f\|_{L^{\infty}},\\
&\displaystyle\left\|\frac{1}{\hbar}[p_{N}(f),p_{N}(g)]-p_{N}\left(\left\{f,g\right\}\right)\right\|_{L_{N}^{\infty}}=O(\hbar).
\end{eqnarray*}
\end{theorem}

Later, Charles and Polterovich~\cite{ChPolt2017} established a sharper estimate: there exists constants $c_0,c_1>0$ such that for all $f,g\in C^3(\mathrm{S}^2)$
\begin{equation}
\label{PoltEstim}
\begin{split}
&\left\|\frac{1}{\hbar}[p_{N}(f),p_{N}(g)]-p_{N}\left(\left\{f,g\right\}\right)\right\|_{L_{N}^{\infty}}\le{} \hbar c_{0}\left(\|f\|_{C^{1}}\|g\|_{C^{3}}+\|f\|_{C^{2}}\|g\|_{C^{2}}+\|f\|_{C^{3}}\|g\|_{C^{1}}\right),{}\\&
\displaystyle\|f\|_{L^{\infty}}-c_1\|f\|_{C^{2}}\hbar\le\|p_{N}(f)\|_{L_{N}^{\infty}}\le\|f\|_{L^{\infty}}
\end{split}
\end{equation}
where $\|f\|_{C^{k}}=\max\limits_{i\le k}\,\sup|\nabla^{i} f|$.

\subsection{Quantized MHD system}
As we can see from \eqref{MHDvort}, the stream functions $\psi,\theta$ and the vorticities $\omega,\beta$ are related to one another through the Laplace-Beltrami operator $\Delta$. 
Therefore, to complete the spatial discretization, we need also to discretize the Laplacian. 
Indeed, the quantized Laplacian on $\mathfrak{su}(N)$ is given by the \emph{Hoppe--Yau Laplacian}~\cite{HoppYau}:
\begin{equation}
\label{DiscrLapl}
\Delta_{N}(\cdot)=\frac{1}{\hbar^2}\left([X_{1}^{N},[X_{1}^{N},\cdot]]+[X_{2}^{N},[X_{2}^{N},\cdot]]+[X_{3}^{N},[X_{3}^{N},\cdot]]\right),
\end{equation}
where 
$X_{a}$, $a=1,2,3$ are generators of a unitary irreducible ``spin $(N-1)/2$'' representation of $\mathfrak{so}(3)$, i.e.,
\begin{equation*}
\frac{1}{\hbar}[X_{a},X_{b}]=\mathrm{i}\varepsilon_{abc} X_{c},\quad X_{1}^{2}+X_{2}^{2}+X_{3}^{2}=\mathbb{I},
\end{equation*}
where $\varepsilon_{abc}$ is the Levi--Civita symbol. 
\new{Whereas the matrix realization of the generators $X_{a}$ is arbitrary, suitable explicit expressions can be found in \cite{HoppYau}.}

The Hoppe--Yau Laplacian \eqref{DiscrLapl} corresponds to the continuous Laplace-Beltrami operator $\Delta$ in the sense that the matrices $T_{lm}^{N}$ are eigenvectors of $\Delta_{N}$ with eigenvalues $-l(l+1)$:
\begin{equation}
\label{eigT}
\Delta_{N}T_{lm}^{N}=-l(l+1)T_{lm}^{N},
\end{equation}
while the spherical harmonics $Y_{lm}$ are eigenvectors of $\Delta$ with the same eigenvalues:
\begin{equation*}
\Delta Y_{lm}=-l(l+1)Y_{lm}.
\end{equation*}

Let us now give an explicit correspondence between the continuous function $\omega\in C_{0}^{\infty}(\mathrm{S}^{2})$ and its quantized counterpart $W\in\mathfrak{su}(N)$. The function $\omega$ can be decomposed in the spherical harmonics basis, $\omega=\sum_{l=1}^{\infty}\sum_{m=-l}^{l}\omega^{lm}Y_{lm}$ and therefore
\begin{equation*}
W=p_{N}(\omega)=\sum\limits_{l=1}^{N-1}\sum\limits_{m=-l}^{l}\mathrm{i}\omega^{lm}T_{lm}^{N}.
\end{equation*}
If the function $\omega\in C_{0}^{\infty}(\mathrm{S}^{2})$ is real-valued, then $\omega^{lm}=(-1)^{m}\omega^{l(-m)}$, which implies that the matrix $W$ is skew-Hermitian:
\begin{equation*}
W+W^{\dagger}=0 \iff W\in\mathfrak{u}(N).
\end{equation*}
Furthermore, since $\omega$ has vanishing circulation, we have $\omega^{0,0} = 0$, which implies that $\operatorname{tr}(W) = 0$, i.e., $W \in \mathfrak{su}(N)$. 
Also, the Hoppe--Yau Laplacian $\Delta_{N}$ restricts to a bijective operator on $\mathfrak{su}(N)$
\begin{equation*}
\Delta_{N}\colon\mathfrak{su}(N)\to\mathfrak{su}(N).
\end{equation*}

We have now all the ingredients to write down the spatially discretized analogue of incompressible MHD equations \eqref{MHDvort} on the sphere, similarly to how it is done for the incompressible Euler equations in the work of \cite{Zeit}. 
Namely, we replace the continuous flow \eqref{MHDvort} with its quantized counterpart:
\begin{equation}
\label{qMHD}
\left\{
\begin{aligned}
&\dot W=\frac{1}{\hbar}[W,\Delta^{-1}_{N}W]+\frac{1}{\hbar}[\Theta,\Delta_{N}\Theta], \\
&\dot\Theta=\frac{1}{\hbar}[\Theta,\Delta^{-1}_{N}W],
\end{aligned}
\right.
\end{equation}
where $W,\Theta\in\mathfrak{su}(N)$.
\begin{remark}
In case of trivial magnetic field, $\Theta=0$, equation \eqref{qMHD} coincides with Zeitlin's model for the incompressible Euler equations on the sphere. 
\end{remark}

Let us now give a different description of the matrices $T_{lm}^{N}$.
Despite that we have the explicit formula~\eqref{proj2} for them, its usage is not efficient due to high computational complexity of the algorithm for finding Wigner 3j symbols. 
Instead, let us note that due to construction \eqref{DiscrLapl} of the Hoppe--Yau Laplacian, it preserves the space of matrices with zero entries except on $\pm m$ off diagonals. 
This allows us to identify the corresponding eigenmatrix $T_{lm}^{N}$ with a sparse skew-hermitian matrix that has non-trivial entries only on the $m$ off diagonal, thus reducing the eigenvalue problem \eqref{eigT} to an $(N-|m|)$-dimensional tridiagonal eigenvalue problem. 
Therefore, the complexity of computing the entire basis $(T_{lm}^{N})$ for fixed $N$ is $O(N^{2})$, instead of $O(N^{4})$ if $\Delta_{N}$ were a full matrix. 
Further, to compute the commutator requires $O(N^{3})$ operations per iteration, which gives the entire complexity of the algorithm as $O(N^{3})$ per iteration.
For details, see \cite{CiViMo2023}.
\subsection{Lie--Poisson nature of the quantized flow}
One essential property of Zeitlin's approach via quantization is that it preserves the Lie--Poisson nature of the flow.
In other words, the quantized flow is a Lie--Poisson system, exactly as the continuous one, but on a finite-dimensional counterpart of $\mathfrak{imh}^{*}$.

Introducing $M_{1}=\Delta^{-1}_{N}W$ and $M_{2}=\Delta_{N}\Theta$, and, for convenience, rescaling time by $\hbar$, we rewrite system \eqref{qMHD} as
\begin{equation}
\label{qEu}
\left\{
\begin{aligned}
&\dot W=[W,M_{1}]+[\Theta,M_{2}], \\
&\dot\Theta=[\Theta,M_{1}].
\end{aligned}
\right.
\end{equation}
We now show that \eqref{qEu} is a Lie--Poisson system on the dual of the Lie algebra $\mathfrak{f}=\mathfrak{su}(N)\ltimes\mathfrak{su}(N)^{*}$.

First, we introduce the \textit{magnetic extension} $F=\mathrm{SU}(N)\ltimes\mathfrak{su}(N)^{*}$ of the group $\mathrm{SU}(N)$. The group operation in $F$ is
\begin{equation*}
(\varphi,a)\cdot(\psi,b)=(\varphi\psi,\mathrm{Ad}^{*}_{\psi}a+b),\quad \varphi,\psi\in\mathrm{SU}(N),\,a,b\in\mathfrak{su}(N)^{*}.
\end{equation*}
The adjoint operator on the Lie algebra $\mathfrak{f}=\mathfrak{su}(N)\ltimes\mathfrak{su}(N)^{*}$ is
\begin{equation*}
\mathrm{ad}_{(v,a)}\colon\mathfrak{f}\to\mathfrak{f},\quad \mathrm{ad}_{(v,a)}(w,b)=([v,w],\mathrm{ad}^{*}_{w}a-\mathrm{ad}^{*}_{v}b)
\end{equation*}
for $v,w\in\mathfrak{su}(N)$, $a,b\in\mathfrak{su}(N)^{*}$. From now on, we will identify the Lie algebra $\mathfrak{su}(N)$ with its dual $\mathfrak{su}(N)^{*}$ via the Frobenius inner product
\begin{equation}
\label{frob}
\langle A,B\rangle=\mathrm{tr}(A^{\dagger}B),\quad A\in\mathfrak{su}(N)^{*},\,B\in\mathfrak{su}(N).
\end{equation}
Then, the dual $\mathfrak{f}^{*}=\left(\mathfrak{su}(N)\ltimes\mathfrak{su}(N)^{*}\right)^{*}\simeq\mathfrak{su}(N)\ltimes\mathfrak{su}(N)^{*}$ can be identified with $\mathfrak{f}$ via the pairing
\begin{equation*}
\langle(\xi,a),(w,b)\rangle=\langle b,\xi\rangle+\langle a,w\rangle,
\end{equation*}
where $\langle\cdot,\cdot\rangle$ is defined by \eqref{frob} for $\xi,w\in\mathfrak{su}(N)$, $a,b\in\mathfrak{su}(N)^{*}$,
and the coadjoint action of $\mathfrak{f}$ on $\mathfrak{f}^{*}$ is
\begin{equation*}
\mathrm{ad}_{(v,a)}^{*}\colon\mathfrak{f}^{*}\to\mathfrak{f}^{*},\quad \mathrm{ad}^{*}_{(v,a)}(w,b)=([w,v],\mathrm{ad}^{*}_{v}b-\mathrm{ad}^{*}_{w}a),
\end{equation*}
where $(v,a)\in\mathfrak{f}$, $(w,b)\in\mathfrak{f}^{*}$. Using \eqref{frob}, one can get an explicit formula for $\mathrm{ad}^{*}$ operator as
\begin{equation*}
\mathrm{ad}^{*}_{(v,a)}(w,b)=([w,v],[v^{\dagger},b]+[a,w^{\dagger}]).
\end{equation*}

Summarizing the above discussion, we arrive at the following result.
\begin{proposition}
System \eqref{qEu} is a Lie--Poisson flow on the dual $\mathfrak{f}^{*}$ of the Lie algebra $\mathfrak{f}=\mathfrak{su}(N)\ltimes\mathfrak{su}(N)^{*}$:
\begin{equation*}
\dot J=\operatorname{ad}^{*}_{M}J,
\end{equation*}
where $J=(\Theta,W^{\dagger})\in\mathfrak{f}^{*}$, $M=(M_{1},M_{2}^{\dagger})\in\mathfrak{f}$, with the Hamiltonian
\begin{equation}
\label{hamilt}
H(W,\Theta)=\frac{1}{2}\left(\mathrm{tr}(W^{\dagger}M_{1})+\mathrm{tr}(\Theta^{\dagger}M_{2})\right).
\end{equation}
\end{proposition}

The Hamiltonian nature of the quantized flow \eqref{qEu} suggests that there are quantized analogues of the Casimirs \eqref{enstr}.
Indeed, they are 
\begin{equation}
\label{qCas1}
\mathcal{C}^N_{f}(i\Theta)=\frac{4\pi}{N}\mathrm{tr}\left(f(\Theta)\right),\quad I^N_{g}(i W,i \Theta)=\frac{4\pi}{N}\mathrm{tr}(Wg(\Theta)),
\end{equation}
for arbitrary smooth functions $f\colon \mathbb{R}\to\mathbb{R}$ and $g\colon \mathbb{R}\to\mathbb{R}$. 

Since the flow \eqref{qEu} is a finite-dimensional Lie--Poisson system, only a finite number of Casimirs are independent. Therefore, it is enough to consider functions $f$ and $g$ being monomials, and the corresponding Casimirs of the quantized MHD flow \eqref{qEu} are
\begin{equation}
\label{qCas}
\mathcal{C}^N_{m}(i\Theta)=\frac{4\pi}{N}\mathrm{tr}\left(\Theta^{m}\right),\quad I^N_{m}(i W,i \Theta)=\frac{4\pi}{N}\mathrm{tr}(W\Theta^{m}),\quad m=1,\ldots,N.
\end{equation}

\subsubsection{Convergence results for quantized Casimirs}
Here, we prove that the quantized Casimirs \eqref{qCas} converge to the corresponding continuous Casimirs
\begin{equation*}
\mathcal{C}_{m}(\theta)=\int\limits_{S^{2}}\theta^{m}\mu,\quad I_{m}(\theta,\omega)=\int\limits_{S^{2}}\omega\theta^{m}\mu.
\end{equation*}

\begin{theorem}\label{thm:casimir_convergence}
  There exist a constant $c_{1}>0$ such that for all $\omega,\theta\in C^2(\mathrm{S}^2)$  we have
  \begin{equation*}
  \begin{split}
    &\lvert \mathcal{C}^N_m(p_N\theta) - \mathcal{C}_m(\theta) \rvert \leq \beta_{1}(\theta,m)\hbar,\\{}&
   \left|I^{N}_{m}(p_{N}\theta,p_{N}\omega)-I_{m}(\theta,\omega)\right|\le4\pi\|\omega\|_{\infty}\beta_{2}(\theta,m)\hbar,
    \end{split}
  \end{equation*}
  where
  \begin{equation*}
  \begin{split}
&\beta_{1}(\theta,m)=c_{1}\sum\limits_{j=2}^{m-1}(4\pi\|\theta\|_{\infty})^{m-j}\left(\|\theta\|_{\infty}\|\theta^{j-1}\|_{C^{2}}+\|\theta\|_{C^{1}}\|\theta^{j-1}\|_{C^{1}}\right),{}\\&
\beta_{2}(\theta,m)=c_{1}\sum\limits_{j=2}^{m}(4\pi\|\theta\|_{\infty})^{m-j}\left(\|\theta\|_{\infty}\|\theta^{j-1}\|_{C^{2}}+\|\theta\|_{C^{1}}\|\theta^{j-1}\|_{C^{1}}\right)
\end{split}
\end{equation*}
  In particular, $\lvert I^{N}_{m}(p_N\theta,p_N\omega) - I_m(\theta,\omega) \rvert = \mathcal{O}(1/N)$, and $\lvert \mathcal{C}^N_m(p_N\theta) - \mathcal{C}_m(\theta) \rvert = \mathcal{O}(1/N)$ as $N\to\infty$.
\end{theorem}
\new{
\begin{remark}
One observes that the rate of spatial convergence is $1/N$, and the discrete model has $N^{2}$ degrees of freedom. 
In other words, the order of convergence is $1/2$.
\end{remark}
}
\begin{proof}
  Set $\mathrm{i}\Theta_N = p_N\theta$. Then
  \begin{align*}
       \mathcal{C}^{N}_m(p_N\theta) =& \frac{4\pi}{N} \operatorname{tr}(\Theta_N^m) = \frac{4\pi}{N}\operatorname{tr}(\Theta_N \Theta_N^{m-1}) = \frac{4\pi}{N}\operatorname{tr}\left( \Theta_N\big(\Theta_N^{m-1}- p_N(\theta^{m-1})/\mathrm{i} + p_N(\theta^{m-1})/\mathrm{i}\big) \right) = \\
      &  \frac{4\pi}{N}\operatorname{tr}\left( \Theta_N  p_N(\theta^{m-1})/\mathrm{i}\right) + \frac{4\pi}{N}\operatorname{tr}\left(\Theta_N\big(\Theta_N^{m-1}-p_N(\theta^{m-1})/\mathrm{i}\big) \right).
  \end{align*}
We thus obtain the following estimate:
\begin{equation}
\label{casconvineq}
\begin{split}
&\left|\mathcal{C}^{N}_{m}(p_{N}\theta)-\frac{4\pi}{N}\operatorname{tr}\left(\Theta_{N}p_{N}(\theta^{m-1})/\mathrm{i}\right)\right|=\left| \frac{4\pi}{N}\operatorname{tr}\left( \Theta_N\big(\Theta_N^{m-1}-p_N(\theta^{m-1})/\mathrm{i}\big)\right)  \right| \leq  {}\\&\underbrace{\left(\frac{1}{N}\sum_{k=1}^N \lvert \lambda_k(\Theta_{N}) \rvert \right)}_{\lVert \Theta_N\rVert_1}\lVert \Theta_N^{m-1} - p_N(\theta^{m-1})/\mathrm{i} \rVert_{L_{N}^{\infty}}\le 4\pi\|\theta\|_{\infty}\lVert \Theta_N^{m-1} - p_N(\theta^{m-1})/\mathrm{i} \rVert_{L_{N}^{\infty}},
\end{split}
\end{equation}
because $\lVert \Theta_N\rVert_1\le 4\pi\|\Theta_{N}\|_{L_{N}^{\infty}}\le4\pi\|\theta\|_{\infty}$, see \cite[lemma 6]{MoPe2024} and the second estimate in \eqref{PoltEstim}. 

Let us first elaborate on the term $\lVert\Theta_N^{m-1} - p_N(\theta^{m-1})/\mathrm{i} \rVert_{L_{N}^{\infty}}$.
\begin{equation}
\label{casconvineq1}
\begin{split}
\|\Theta_{N}^{m-1}-p_{N}(\theta^{m-1})/\mathrm{i}\|_{L_{N}^{\infty}}=&\|\Theta_{N}\Theta_{N}^{m-2}-p_{N}(\theta^{m-1})/\mathrm{i}+\Theta_{N} p_{N}(\theta^{m-2})/\mathrm{i}-\Theta_{N}p_{N}(\theta^{m-2})/\mathrm{i}\|_{L_{N}^{\infty}}\le{}\\&\|\Theta_{N}(\Theta_{N}^{m-2}-p_{N}(\theta^{m-2})/\mathrm{i})\|_{L_{N}^{\infty}}+\|p_{N}(\theta^{m-1})-\Theta_{N} p_{N}(\theta^{m-2})\|_{L_{N}^{\infty}}\le{}\\&\|\Theta_{N}\|_{1}\|\Theta_{N}^{m-2}-p_{N}(\theta^{m-2})/\mathrm{i}\|_{L_{N}^{\infty}}+\|p_{N}(\theta^{m-1})-\Theta_{N}p_{N}(\theta^{m-2})\|_{L_{N}^{\infty}}\leq{}\\&4\pi\|\theta\|_{\infty}\|\Theta_{N}^{m-2}-p_{N}(\theta^{m-2})/\mathrm{i}\|_{L_{N}^{\infty}}+\|p_{N}(\theta^{m-1})-\Theta_{N}p_{N}(\theta^{m-2})\|_{L_{N}^{\infty}}.
\end{split}
\end{equation}
Introducing the notation 
\begin{equation*}
J(m-1)=\|\Theta_{N}^{m-1}-p_{N}(\theta^{m-1})/\mathrm{i}\|_{L_{N}^{\infty}},\quad K(m-1)=\|p_{N}(\theta^{m-1})-\Theta_{N}p_{N}(\theta^{m-2})\|_{L_{N}^{\infty}},\quad A=4\pi\|\theta\|_{\infty},
\end{equation*}
we can write \eqref{casconvineq1} as
\begin{equation*}
J(m-1)\le AJ(m-2)+K(m-1),
\end{equation*}
and therefore
\begin{equation*}
J(m-1)\le \sum\limits_{j=2}^{m-1}A^{m-j-1}K(j),
\end{equation*}
which implies that \eqref{casconvineq} reads
\begin{equation*}
\left|\mathcal{C}^{N}_{m}(p_{N}\theta)-\frac{4\pi}{N}\operatorname{tr}\left(\Theta_{N}p_{N}(\theta^{m-1})/\mathrm{i}\right)\right|\le\sum\limits_{j=2}^{m-1}A^{m-j}K(j)
\end{equation*}
We observe that due to estimate in \cite[prop. 3.12]{ChPolt2017}, it follows that 
\begin{equation}
\label{estimJK}
K(j)=\|p_{N}(\theta^{j})-\Theta_{N}p_{N}(\theta^{j-1})\|_{L_{N}^{\infty}}\le c_{1}\hbar\left(\|\theta\|_{\infty}\|\theta^{j-1}\|_{C^{2}}+\|\theta\|_{C^{1}}\|\theta^{j-1}\|_{C^{1}}\right),
\end{equation}
and we finally conclude that
\begin{equation}
\label{estim1CfN}
\left|\mathcal{C}^{N}_{m}(p_{N}\theta)-\frac{4\pi}{N}\operatorname{tr}\left(\Theta_{N}\frac{p_{N}(\theta^{m-1})}{\mathrm{i}}\right)\right|\le\beta_{1}(\theta,m)\hbar,\quad N\to\infty,
\end{equation}
where $\beta_{1}(\theta,m)\ge0$ is determined via \eqref{estimJK}:
\begin{equation*}
\beta_{1}(\theta,m)=c_{1}\sum\limits_{j=2}^{m-1}(4\pi\|\theta\|_{\infty})^{m-j}\left(\|\theta\|_{\infty}\|\theta^{j-1}\|_{C^{2}}+\|\theta\|_{C^{1}}\|\theta^{j-1}\|_{C^{1}}\right),
\end{equation*}
where $c_{1}>0$ is a constant.

As a by product, we also get that
\begin{equation}
\label{thetaconverg}
\lVert\Theta_N^{m} - p_N(\theta^{m})/\mathrm{i} \rVert_{L_{N}^{\infty}}\le\beta_{2}(\theta,m)\hbar=\mathcal{O}(1/N),\quad N\to\infty,
\end{equation}
where
\begin{equation*}
\beta_{2}(\theta,m)=c_{1}\sum\limits_{j=2}^{m}(4\pi\|\theta\|_{\infty})^{m-j}\left(\|\theta\|_{\infty}\|\theta^{j-1}\|_{C^{2}}+\|\theta\|_{C^{1}}\|\theta^{j-1}\|_{C^{1}}\right)
\end{equation*}
for any $m\ge2$.

  Further, since the scaled Frobenius inner product $\frac{4\pi}{N}\operatorname{tr}(AB)$ on $i\mathfrak{u}(N)$ corresponds to the $L^2$ inner product, i.e., $\langle p_N^*iA,p_N^*iB \rangle_{L^2} = \frac{4\pi}{N}\operatorname{tr}(AB) =: \langle A,B\rangle_{L^2_N}$ for the $L^2$ isometric embedding $p_N^*\colon \mathfrak{u}(N)\to L^2(\mathrm{S}^2)$, we have 
  \begin{equation*}
      \lim_{N\to\infty} \frac{4\pi}{N} \operatorname{tr}\left(\Theta_N  \frac{p_N(\theta^{m-1})}{\mathrm{i}}\right) =\langle \theta,\theta^{m-1}\rangle_{L^2} = \mathcal{C}_{m}(\theta),    
  \end{equation*}
  moreover, due to estimates in \cite[lemma 5]{MoPe2024}, we have that
  \begin{equation}
  \label{estim2CfN}
  \left|\mathcal{C}_{m}(\theta)-\frac{4\pi}{N}\operatorname{tr}\left(\Theta_{N}\frac{p_{N}(\theta^{m-1})}{\mathrm{i}}\right)\right|\le2\hbar^{2}\|\theta\|_{H^{2}(S^{2})}\|\theta^{m-1}\|_{H^{2}(S^{2})}=\beta_{3}(\theta)\hbar^{2},
\end{equation}
where $\beta_{3}(\theta,m)=2\|\theta\|_{H^{2}(S^{2})}\|\theta^{m-1}\|_{H^{2}(S^{2})}$
  and therefore summing \eqref{estim1CfN} and \eqref{estim2CfN}, we get that for sufficiently large $N$ 
  \begin{equation*}
  \begin{split}
  \left|\mathcal{C}_{m}^{N}(p_{N}\theta)-\mathcal{C}_{m}(\theta)\right|\le&\left|\mathcal{C}^{N}_{m}(p_{N}\theta)-\frac{4\pi}{N}\operatorname{tr}\left(\Theta_{N}\frac{p_{N}(\theta^{m-1})}{\mathrm{i}}\right)\right|+{}\\&\left|\frac{4\pi}{N}\operatorname{tr}\left(\Theta_{N}\frac{p_{N}(\theta^{m-1})}{\mathrm{i}}\right)-\mathcal{C}_{m}(\theta)\right|\le\beta_{1}(\theta,m)\hbar+\beta_{3}(\theta,m)\hbar^{2}=\mathcal{O}(1/N),
  \end{split}
  \end{equation*}
which finally proves $\mathcal{O}(1/N)$ convergence of $\mathcal{C}_{m}^{N}(p_{N}\theta)$ as $N\to\infty$.

Now we prove a similar result for the cross-helicity Casimir $I_{m}$. 
\begin{equation*}
\begin{split}
I_{m}^{N}(p_{N}\theta,p_{N}\omega)=\frac{4\pi}{N}\operatorname{tr}(W_{N}\Theta_{N}^{m})=\frac{4\pi}{N}\operatorname{tr}\left(W_{N}\frac{p_{N}(\theta^{m})}{\mathrm{i}}\right)+\frac{4\pi}{N}\operatorname{tr}\left(W_{N}\left(\Theta_{N}^{m}-\frac{p_{N}(\theta^{m})}{\mathrm{i}}\right)\right).
\end{split}
\end{equation*}
For sufficiently large $N$,
\begin{equation}
\label{convhelic1}
\left|\frac{4\pi}{N}\operatorname{tr}\left(W_{N}\frac{p_{N}(\theta^{m})}{\mathrm{i}}\right)-I_{m}(\theta,\omega)\right|\le2\hbar^{2}\|\omega\|_{H^{2}(S^{2})}\|\theta^{m}\|_{H^{2}(S^{2})}=\beta_{4}(\theta,\omega)\hbar^{2},
\end{equation}
where $\beta_{4}(\theta,\omega)=2\|\omega\|_{H^{2}(S^{2})}\|\theta^{m}\|_{H^{2}(S^{2})}$.

Further,
\begin{equation}
\label{convhelic2}
\begin{split}
&\left|I_{m}^{N}(p_{N}\theta,p_{N}W)-\frac{4\pi}{N}\operatorname{tr}\left(W_{N}\frac{p_{N}(\theta^{m})}{\mathrm{i}}\right)\right|=\left|\frac{4\pi}{N}\operatorname{tr}\left(W_{N}\left(\Theta_{N}^{m}-\frac{p_{N}(\theta^{m})}{\mathrm{i}}\right)\right)\right|\le{}\\&\underbrace{\left(\frac{1}{N}\sum_{k=1}^N \lvert \lambda_k(W_{N}) \rvert \right)}_{\lVert W_N\rVert_1}\lVert \Theta_N^{m} - p_N(\theta^{m})/\mathrm{i} \rVert_{L_{N}^{\infty}}\le4\pi\|\omega\|_{\infty}\beta_{2}(\theta,m)\hbar
\end{split}
\end{equation}
due to \eqref{thetaconverg}, and finally summing \eqref{convhelic1} and \eqref{convhelic2}, we get
\begin{equation*}
\left|I^{N}_{m}(p_{N}\theta,p_{N}\omega)-I_{m}(\theta,\omega)\right|\le4\pi\|\omega\|_{\infty}\beta_{2}(\theta,m)\hbar+\beta_{4}(\theta,\omega)\hbar^{2}=\mathcal{O}(1/N)
\end{equation*}
for $N\to\infty$.
\end{proof}

\begin{remark}
Preservation of the Casimirs $\mathcal{C}_{m}^{N}(\Theta)$ \new{for $m=1,\ldots,N$} is equivalent to preservation of the spectrum of $\Theta$.
\new{Indeed, this follows from the relations $\sum_{k=1}^N \lambda_k(\Theta)^m = \operatorname{tr}(\Theta^m)$.}
\end{remark}

\section{Lie--Poisson preserving time integrator}\label{sec:temporal_discretization}
To get the fully discretized incompressible MHD equations, one also needs to discretize system \eqref{qEu} in time. 
There are generic time integration methods for Lie--Poisson systems (see, e.g., \cite{BoMa2016}), but these make heavy use of the matrix exponential.
Such methods are computationally too expensive when the dimension of the Lie algebra is large (as in the case here).
Our goal is instead to develop a ``matrix exponential free'' integrator that preserves the underlying Lie--Poisson geometry of the flow \eqref{qEu}, meaning it should preserve the Casimirs \eqref{qCas} exactly, be a symplectic map on the coadjoint orbits of $\mathfrak{f}^{*}$, and thereby nearly preserves the Hamiltonian \eqref{hamilt} in the sense of backward error analysis \cite{HaiLubWan}.

There are several ways to construct symplectic integrators for Hamiltonian systems on $T^*\mathbb{R}^n$, among them are \emph{symplectic Runge-Kutta methods}~\cite{Sa1988}. 
Given a Butcher tableau
\begin{equation*}
\begin{array}
{c|cccc}
c_{1} & a_{11} & a_{12} &\cdots & a_{1s}\\
c_{2} &a_{21} & a_{22} &\cdots & a_{2s}\\
\vdots &\vdots & \vdots &\ddots & \vdots \\
c_{s}& a_{s1} & a_{s2} &\cdots & a_{ss}\\
\hline
& b_{1} & b_{2} &\cdots & b_{s}
\end{array}
\end{equation*}
with $b_{i}a_{ij}+b_{j}a_{ji}=b_{i}b_{j}$ for all $i,j=1,\ldots,s$, the corresponding method being applied to Hamiltonian systems on a symplectic vector space $(\mathbb{R}^{2n},\Omega)$ is symplectic. An example is the implicit midpoint method. 
However, when directly applied to a Lie--Poisson system, a symplectic Runge-Kutta scheme does not yield a Poisson integrator.

There exist a few approaches to obtain Poisson integrators for Lie--Poisson systems $(P,\left\{\cdot,\cdot\right\},H)$.
\begin{itemize}
\item If $P=\mathfrak{g}^{*}$, and the Hamiltonian $H$ can be split into the sum of integrable Hamiltonians, $H=\sum H_{i}$, one can use splitting methods \cite{McQui1,McQui2}.
\item If $P=\mathfrak{g}^{*}$, the Lie--Poisson system is a Poisson reduction of a Hamiltonian system on $T^{*}G$. In this case, the discrete Lie--Poisson flow is constructed from a discrete $G$-invariant Lagrangian on $TG$ \cite{ChanScov,MarsPekSh}. The other approach is to embed $G$ in a linear space and use constrained symplectic integrator RATTLE \cite{ChanScov1,Jay,McLModVerdWil}. This, however, results in a very complicated scheme on high dimensional vector spaces. For example, in case of a 2-dimensional sphere $\mathrm{S}^{2}$, which is a coadjoint orbit of $\mathfrak{so}(3)^{*}$, one would lift the equations to $T^{*}\mathrm{SO}(3)$ embedded in $T^{*}\mathbb{R}^{3\times 3}$ with dimension 18.
\item For domains originating from the generalized Hopf fibration, one can use \textit{collective symplectic integrators}. See \cite{McLachModVerd} for details.
\end{itemize}

The other approach is to make use of the Poisson reduction (more precisely, Poisson reconstruction) to reduce the discrete symplectic flow on $T^{*}G$ (for example, symplectic Runge-Kutta method) to a discrete Lie--Poisson flow on $\mathfrak{g}^{*}$ \cite{ModViv}. This is how \textit{isospectral Runge-Kutta methods} were developed for a large class of isospectral flows on $J$-quadratic Lie algebras, including the Euler-Zeitlin equations on a sphere. The main advantages of the method are that it is formulated directly on the algebra, does not involve expensive group-to-algebra maps, and can be applied to any isospectral flow. Therefore, we might expect that using the strategy from \cite{ModViv} to construct a Lie--Poisson integrator for \eqref{qEu} will give the same benefits, as isospectral flows considered in \cite{ModViv} have a similar geometry as equations \eqref{qEu} do. 

We mention also the work of Kraus, Tassi, and Grasso \cite{KrTassGrass2016}, where an integrator for 2-D MHD on the plane is developed. 
The integrator preserves the linear and quadratic Casimirs and the energy of MHD equations on the plane. 
However, the method does not preserve higher order Casimir, nor the Lie--Poisson structure.

As we shall see below, the strategy of using the Poisson reduction results in the numerical scheme for incompressible MHD on the sphere that completely preserves the underlying geometry of the equations. 

\subsection{Matrix representation of $\mathfrak{su}(N)\ltimes\mathfrak{su}(N)^{*}$}
The first natural attempt to derive structure preserving integrator for \eqref{qEu} is to represent it as an isospectral flow on a space of $2N\times 2N$ matrices, in other words, to convert the system of matrix equations \eqref{qEu} into a single matrix flow. 
That could potentially make it possible to apply the isospectral integrators developed in \cite{ModViv}.

Let us introduce the two lower triangular block matrices
\begin{equation}
\label{semdirrepr}
V=\begin{pmatrix}
\Theta & 0\\
W & \Theta
\end{pmatrix},\quad
M=\begin{pmatrix}
M_{1} & 0\\
M_{2} & M_{1}
\end{pmatrix}.
\end{equation}
This embeds $\mathfrak{f} = \mathfrak{su}(N)\ltimes \mathfrak{su}(N)^*$ as a subalgebra of $\mathfrak{gl}(2N,\mathbb{C})$ such that the equations \eqref{qEu} constitute an isospectral flow of matrices of the form \eqref{semdirrepr}:
\begin{equation}
\label{isoMHD}
\dot V=[V,M(V)].
\end{equation}

Let us check whether $\mathfrak{f}\subset \mathfrak{gl}(2N,\mathbb{C})$ fits the conditions stated in \cite{ModViv} for isospectral symplectic Runge-Kutta integrators to work, i.e., that it is $J$-\textit{quadratic} and \emph{reductive}.

\begin{definition}
  Let $J$ be a matrix such that $J^2 = c I$, where $I$ is the identity matrix.
  The corresponding $J$-\textit{quadratic} Lie algebra $\mathfrak{g}\subset\mathfrak{gl}(N,\mathbb{C})$ is given by
  \begin{equation*}
  A \in \mathfrak{g} \iff A^{\dagger}J+JA=0.
  \end{equation*}
\end{definition}

\begin{lemma}\label{lemmaMHDiso}
  Let $\mathfrak{g}\subset\mathfrak{gl}(N,\mathbb{C})$ be $J$-quadratic. 
  Then the Lie algebra $\mathfrak{g}\ltimes\mathfrak{g}^{*}\subset\mathfrak{gl}(2N,\mathbb{C})$ is a subalgebra of the $\tilde{J}$-quadratic Lie algebra $\tilde{\mathfrak{g}}$ for
  \begin{equation}
  \label{Jmatrix}
  \tilde{J}=\begin{pmatrix}
  0 & J\\
  J & 0
  \end{pmatrix} .
  \end{equation}
\end{lemma}

\begin{proof}
  Clearly, $\tilde{J}^{2}=cI$. 
  Let $A\in\mathfrak{g}\ltimes\mathfrak{g}^* \subset \mathfrak{gl}(2N, \mathbb{C})$, i.e.
  \begin{equation*}
  A=\begin{pmatrix}
  \Theta & 0\\
  W & \Theta
  \end{pmatrix},
  \end{equation*}
  where $W,\Theta\in\mathfrak{g}$. Since $\mathfrak{g}$ is $J$-quadratic, we have
  \begin{equation*}
  \begin{aligned}
  &\Theta^{\dagger}J+J\Theta=0\Longrightarrow\Theta^{\dagger}=-\frac{1}{c}J\Theta J,\\
  &W^{\dagger}J+JW=0\Longrightarrow W^{\dagger}=-\frac{1}{c}JWJ.
  \end{aligned}
  \end{equation*}
  We aim to prove that $A^{\dagger}\tilde{J}+\tilde{J}A=0$. 
  First, we get
  \begin{equation*}
  A^{\dagger}=-\frac{1}{c}\begin{pmatrix}
  J\Theta J & JWJ\\
  0 & J\Theta J
  \end{pmatrix},
  \end{equation*}
  and therefore
  \begin{equation*}
  A^{\dagger}\tilde{J}+\tilde{J}A=-\frac{1}{c}\begin{pmatrix}
  J\Theta J & JWJ\\
  0 & J\Theta J
  \end{pmatrix}\begin{pmatrix}
  0 & J\\
  J & 0
  \end{pmatrix}+\begin{pmatrix}
  0 & J\\
  J & 0
  \end{pmatrix}\begin{pmatrix}
  \Theta & 0\\
  W & \Theta
  \end{pmatrix}=0.
  \end{equation*}
  This concludes the proof.
\end{proof}

At first glance, this result indicates that the isospectral flow \eqref{isoMHD} is suitable for isospectral integrators, in particular, the midpoint isospectral integrator \cite{ModViv,Milo}
\begin{equation}
\label{MHDisoInt}
\begin{aligned}
&V_{n}=\left(I+\frac{h}{2}M(\tilde V)\right)\tilde V\left(I-\frac{h}{2}M(\tilde V)\right),\\
&V_{n+1}=V_{n}+h[\tilde V,M(\tilde V)],
\end{aligned}
\end{equation}
because $\mathfrak{g}=\mathfrak{su}(N)$ is $J$-quadratic with $J=I$. 
More precisely, 
\begin{theorem}
\label{theorem2}
The scheme \eqref{MHDisoInt} constitutes an isospectral integrator for the isospectral flow \eqref{isoMHD}. It preserves the Casimirs 
\begin{equation}
\label{casV}
\mathrm{tr}(V_{n}^{k})=\mathrm{tr}(V_{n+1}^{k}).
\end{equation}
\end{theorem}
\begin{proof}
A direct consequence of Lemma \ref{lemmaMHDiso} and Theorem 1 in \cite{ModViv}.
\end{proof}

However, this result is not enough, since $\mathfrak{g}\ltimes\mathfrak{g}^*$ is a proper subalgebra of the larger $\tilde J$-quadratic algebra $\tilde{\mathfrak{g}}$.
Indeed, there is no guarantee that $V$ remains of the lower triangular block form \eqref{semdirrepr} as the general form of an element in $\tilde{\mathfrak{g}}$ is
\begin{equation*}
V=\begin{pmatrix}
\Theta & B\\
W & \Theta
\end{pmatrix},
\end{equation*}
where $W,\Theta,B\in\mathfrak{g}$.

\begin{remark}
\label{rem1}
Assuming that $V$ remains of the lower triangular block form \eqref{semdirrepr},
Theorem \ref{theorem2} explains preservation of the spectrum of $\Theta$, as it follows directly from the formula \eqref{casV} since $\operatorname{tr}(V^k) = 2\operatorname{tr}(\Theta^k)$ in this case. 
However, we cannot obtain preservation of the cross-helicity Casimir from \eqref{casV}, again because $\tilde{\mathfrak{g}}$ is a larger Lie algebra than $\mathfrak{g}\ltimes\mathfrak{g}^*$.
Moreover, the Lie algebra $\mathfrak{f} = \mathfrak{g}\ltimes\mathfrak{g}^*$ is not \textit{reductive}, which means that the condition
\begin{equation*}
[\mathfrak{f}^{\dagger},\mathfrak{f}]\subseteq\mathfrak{f}
\end{equation*}
does not hold. 
Consequently, whether the flow preserves the Lie--Poisson structure cannot be addressed with the method developed in \cite{ModViv}, since that method requires a reductive Lie algebra.
But, as we shall see below, the method still preserves all the geometric properties:
\begin{itemize}
\item the scheme \eqref{MHDisoInt} preserves the cross-helicity Casimir;
\item the scheme \eqref{MHDisoInt} is a Lie--Poisson integrator on the dual $\mathfrak{f}^{*}$ of the Lie algebra $\mathfrak{f}=\mathfrak{su}(N)\ltimes\mathfrak{su}(N)^{*}$.
\end{itemize}
Thus, the condition that the Lie algebra be reductive is sufficient, but not necessary, for the isospectral Runge-Kutta integrators developed in \cite{ModViv} to yield a Lie--Poisson integrator.
\end{remark}

The numerical scheme \eqref{MHDisoInt} results in an integrator written for matrices $(W,\Theta)$ in \eqref{qEu} as follows:
\begin{equation}
\label{IsoMHDWtheta}
\begin{aligned}
&\Theta_{n}=\tilde\Theta-\frac{h}{2}[\tilde\Theta,\tilde{M_{1}}]-\frac{h^{2}}{4}\tilde{M_{1}}\tilde\Theta\tilde{M_{1}}, \\
&\Theta_{n+1}=\Theta_{n}+h[\tilde\Theta,\tilde{M_{1}}],\\
&W_{n}=\tilde W-\frac{h}{2}[\tilde W,\tilde{M_{1}}]-\frac{h}{2}[\tilde\Theta,\tilde M_{2}]-\frac{h^{2}}{4}\left(\tilde{M_{1}}\tilde W\tilde{M_{1}}+\tilde M_{2}\tilde\Theta\tilde M_{1}+\tilde M_{1}\tilde\Theta\tilde M_{2}\right), \\
&W_{n+1}=W_{n}+h[\tilde W,\tilde{M_{1}}]+h[\tilde\Theta,\tilde{M_{2}}],
\end{aligned}
\end{equation}
where $\tilde{M_{1}}=\Delta_{N}^{-1}(\tilde W)$ and $\tilde{M_{2}}=\Delta_{N}(\tilde\Theta)$.

In the forthcoming sections we present an alternative derivation of the scheme \eqref{IsoMHDWtheta}, directly using reduction theory for semi-direct products. 
This approach explains the properties of the method \eqref{IsoMHDWtheta} listed in Remark~\ref{rem1}.

\subsection{Reduction theory for semi-direct products}

The strategy of deriving the structure preserving numerical scheme for \eqref{qEu} is based on the following observation. 
The flow \eqref{qEu} on the dual of $\mathfrak{f}=\mathfrak{su}(N)\ltimes\mathfrak{su}(N)^{*}$ can be seen as a Poisson reduction of a Hamiltonian system on $T^{*}F$ with a right-invariant Hamiltonian.
The reduction emerges from the momentum map $\mu\colon T^{*}F\to\mathfrak{f}^{*}$. 
The situation reflects the fact that the continuous equations \eqref{MHD} are a Poisson reduced Hamiltonian flow on the continuous counterpart of the cotangent bundle $T^{*}F$, as was discussed above. 
The momentum map has the property that it is a Poisson map between $T^{*}F$ and $\mathfrak{f}^{*}$. Therefore, having a discrete symplectic flow $\Phi_{h}\colon T^{*}F\to T^{*}F$ that is equivariant with respect to the lifted right action of $F$ on $T^{*}F$, one gets a Poisson integrator $\phi_{h}\colon\mathfrak{f}^{*}\to\mathfrak{f}^{*}$ by applying the momentum map. 

We need therefore to \textit{reconstruct} the canonical system on $T^{*}F$ from the system \eqref{qEu}, apply a symplectic integrator that keeps the flow on $T^{*}F$, check that it is also equivariant, and finally reduce the method back to $\mathfrak{f}^{*}$. 
To do so, one needs the momentum map $\mu$.

First, the cotangent bundle $T^{*}F$ of the magnetic extension $F=\mathrm{SU}(N)\ltimes\mathfrak{su}(N)^{*}$ is
\begin{equation*}
T^{*}F=\left\{(Q,m,P,\alpha)\mid Q\in\mathrm{SU}(N),P\in T^{*}_{Q}(\mathrm{SU}(N)),m\in\mathfrak{su}(N)^{*},\alpha\in\mathfrak{su}(N)\right\}.
\end{equation*}

The lifted left action of the group $F=\mathrm{SU}(N)\ltimes\mathfrak{su}(N)^{*}$ on its cotangent bundle $T^{*}F$ is
\begin{equation}
\label{leftliftact}
(G,u)\cdot(Q,m,P,\alpha)=(GQ,\mathrm{Ad}^{*}_{Q}u+m, (G^{-1})^{\dagger}P,\alpha)
\end{equation}
for $(G,u)\in F$. 
Now, the momentum map associated to the lifted left action \eqref{leftliftact} is given by \cite{MarsRatWein,MarsRat}
\begin{equation}
\label{mommap}
\mu(Q,m,P,\alpha)=\left(\frac{PQ^{\dagger}-QP^{\dagger}}{2},\,Q\alpha Q^{\dagger}\right)=(W^{\dagger},\Theta).
\end{equation}
\begin{proposition}
The canonical equations on $T^{*}F$
\begin{equation}
\label{hamsys}
\left\{
\begin{aligned}
&\dot Q=-M_{1}Q, \\
&\dot P=M_{1}^{\dagger}P+2M_{2}^{\dagger}Q\alpha^{\dagger},\\
&\dot\alpha=0,
\end{aligned}
\right.
\end{equation}
for the right-invariant Hamiltonian $\tilde{H}=H\circ\mu$ defined by
\begin{equation*}
M_{1}=\Delta^{-1}_{N}W,\quad M_{2}=\Delta_{N}\Theta,\quad H(W,\Theta)=\frac{1}{2}\left(\mathrm{tr}(W^{\dagger}M_{1})+\mathrm{tr}(\Theta^{\dagger}M_{2})\right),
\end{equation*}
are reduced to the Lie--Poisson system on $\mathfrak{f}^{*}$
\begin{equation}
\label{redMHD}
\dot W=[W,M_{1}]+[\Theta,M_{2}], \quad \dot\Theta=[\Theta,M_{1}],
\end{equation}
by means of the momentum map \eqref{mommap}.
\end{proposition}
\begin{proof}
First, we observe that since $W\in\mathfrak{su}(N)$, $\Theta\in\mathfrak{su}(N)$, and $\Delta_{N}\colon\mathfrak{su}(N)\to\mathfrak{su}(N)$ by construction \eqref{mommap}, we have that $M_{1},M_{2}\in\mathfrak{su}(N)$. Using \eqref{mommap} we get
\begin{equation*}
\dot W=\frac{1}{2}\left(\dot QP^{\dagger}+Q\dot P^{\dagger}-\dot P Q^{\dagger}-P\dot Q^{\dagger}\right).
\end{equation*}
Using \eqref{hamsys} and rearranging the terms we get
\begin{equation*}
\dot W=\frac{1}{2}\left([QP^{\dagger},M_{1}]+[PQ^{\dagger},M_{1}^{\dagger}]\right)+[Q\alpha Q^{\dagger},M_{2}].
\end{equation*}
Since $M_{1},M_{2}\in\mathfrak{su}(N)$, we have $M_{1}^{\dagger}=-M_{1}$, $M_{2}^{\dagger}=-M_{2}$, and get the first equation in \eqref{redMHD}.
For the $\Theta$ component, we have
\begin{equation*}
\frac{d}{dt}\left(Q\alpha Q^{\dagger}\right)=\dot Q\alpha Q^{\dagger}+Q\alpha\dot Q^{\dagger}.
\end{equation*}
Again, using \eqref{hamsys}, we get
\begin{equation*}
\frac{d}{dt}\left(Q\alpha Q^{\dagger}\right)=[Q\alpha Q^{\dagger},M_{1}]\Longrightarrow \dot\Theta=[\Theta,M_{1}].
\end{equation*}
This concludes the proof.
\end{proof}

\begin{remark}
  We have not written the equation for the variable $m$ in \eqref{hamsys}, as it is not needed to get the algebra variables $(W,\Theta)$ back, see \eqref{mommap}.
\end{remark}

\begin{remark}
  Since, according to \eqref{hamsys}, the matrix $\alpha$ is constant, the equations \eqref{hamsys} can be seen as a Hamiltonian flow on $T^{*}\mathrm{SU}(N)$ with the matrix $\alpha$ as a parameter defining an initial condition for the matrix $\Theta\in\mathfrak{su}(N)$. 
  Thus, despite the incompressible MHD equations have twice as many unknowns as in the incompressible Euler equations, the Hamiltonian left-reconstructed flow still takes place on the cotangent bundle $T^{*}\mathrm{SU}(N)$, which reflects a more general observation that any Lie--Poisson system on a semi-direct product can be viewed as a Newton system with a smaller symmetry group (see \cite{KhMisMod} for details).  
\end{remark}

Recall that the space $\mathfrak{f}^{*}$ can be seen as a quotient of $T^{*}F$ with respect to the lifted right action of $F$, i.e., $\mathfrak{f}^{*}=T^{*}F/F$. 
In other words, points in $\mathfrak{f}^{*}$ are $F$-orbits of points in $T^{*}F$. Then, if points $a\in T^{*}F$ and $b\in T^{*}F$ belong to the same orbit, i.e., $b=f\cdot a$ for some $f\in F$, they correspond to the same point $\mu(a)=\mu(f(a))\in\mathfrak{f}^{*}$. 
Let $\Phi_{h}\colon T^{*}F\to T^{*}F$ be a symplectic method on $T^{*}F$. 
Then it descends to an integrator on $\mathfrak{f}^{*}$ if the points $\Phi_{h}(a)$ and $(\Phi_{h}\circ f)(a)$ belong to the same orbit, which means that $(\Phi_{h}\circ f)(a)=(f\circ \Phi_{h})(a)$. This holds for any point $a\in T^{*}F$, meaning that the method $\Phi_{h}\colon T^{*}F\to T^{*}F$ must also be \textit{equivariant}. 
The setup is illustrated in a diagram Fig.~\ref{equivfig}.

In summary, we arrive at the following result:

\begin{figure}[ht!]
\centering
\includegraphics[scale=0.6]{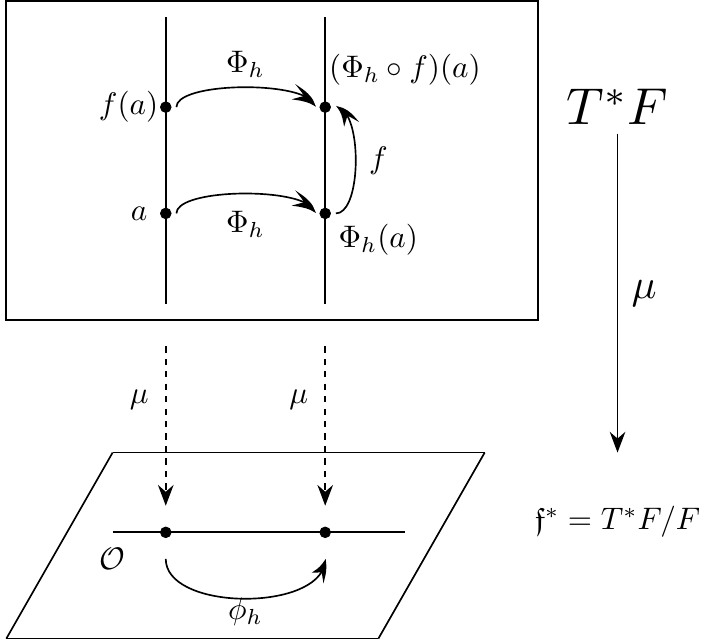}
\caption{Equivariance of a symplectic method $\Phi_{h}\colon T^{*}F\to T^{*}F$. 
The symplectic equivariant method $\Phi_{h}\colon T^{*}F\to T^{*}F$ descends to a Lie--Poisson method $\phi_{h}\colon\mathfrak{f}^{*}\to\mathfrak{f}^{*}$ on the coadjoint orbit $\mathcal{O}\subset\mathfrak{f}^{*}$.}
\label{equivfig}
\end{figure}

\begin{theorem}
\label{th1}
Consider the Lie--Poisson system \eqref{qEu} evolving on the dual of the semi-direct product Lie algebra $\mathfrak{f}^{*}=\mathfrak{su}(N)\ltimes\mathfrak{su}(N)^{*}$. Let $\Phi_{h}\colon T^{*}F\to T^{*}F$ be a symplectic numerical method applied to the Hamiltonian system \eqref{hamsys}. If it is also equivariant with respect to the right $\mathrm{SU}(N)\ltimes\mathfrak{su}(N)^{*}$ action
\begin{equation}
\label{rightAct}
(Q,P,\alpha,m)\cdot(G,u)=\left(QG,P(G^{-1})^{\dagger},\mathrm{Ad}_{G^{-1}}\alpha,\mathrm{Ad}^{*}_{G}m+u\right),
\end{equation}
then it descends to a Lie--Poisson integrator $\phi_{h}$ on $\mathfrak{f}^{*}$.
\end{theorem}

\subsection{Casimir preserving scheme}

According to Theorem~\ref{th1}, we need a symplectic integrator for the Hamiltonian system \eqref{hamsys}. 
We choose the simplest one among symplectic Runge-Kutta methods, which is the implicit midpoint method. 
If we denote the right-hand sides of \eqref{hamsys} by
\begin{equation*}
\begin{aligned}
f(Q,P)&=-M_{1}Q, \\
g(Q,P,\alpha)&=M_{1}^{\dagger}P+2M_{2}^{\dagger}Q\alpha^{\dagger},
\end{aligned}
\end{equation*}
so that
\begin{equation*}
\dot Q=f(Q,P),\quad\dot P=g(Q,P,\alpha),
\end{equation*}
\new{with $\alpha$ being a constant matrix.}

Then the method $\Phi_{h}\colon (Q_{n},P_{n})\mapsto (Q_{n+1},P_{n+1})$ is
\begin{equation}
\label{implMP}
\begin{cases}
\displaystyle Q_{n}=\tilde{Q}-\frac{h}{2}f(\tilde{P},\tilde{Q}),
\quad Q_{n+1}=\tilde{Q}+\frac{h}{2}f(\tilde{P},\tilde{Q})
   \\
   \\
\displaystyle P_{n}=\tilde{P}-\frac{h}{2}g(\tilde{P},\tilde{Q},\alpha),
\quad P_{n+1}=\tilde{P}+\frac{h}{2}g(\tilde{P},\tilde{Q},\alpha).
 \end{cases}
\end{equation}
The implicit midpoint method is known to be symplectic, the only thing we need to prove is that it is also equivariant with respect to action \eqref{rightAct}.
\begin{lemma}
\label{lemm1}
The implicit midpoint method $\Phi_{h}\colon(Q_{n},P_{n})\mapsto(Q_{n+1},P_{n+1})$ defined by \eqref{implMP} is equivariant with respect to action \eqref{rightAct}, i.e.,
\begin{equation*}
\Phi_{h}\circ(G,u)=(G,u)\circ\Phi_{h}.
\end{equation*}
\end{lemma}
\begin{proof}
First, we can write $\Phi_{h}=\Phi_{h}^{(2)}\circ\Phi_{h}^{(1)}$, where
\begin{equation*}
\begin{aligned}
\Phi_{h}^{(1)}&\colon(Q_{n},P_{n})\mapsto(\tilde Q,\tilde P),\\
\Phi_{h}^{(2)}&\colon(\tilde Q,\tilde P)\mapsto(Q_{n+1},P_{n+1}).
\end{aligned}
\end{equation*}
Since equivariance of both $\Phi_{h}^{(1)}$ and $\Phi_{h}^{(2)}$ implies that their composition $\Phi_{h}$ is also equivariant, it is enough to prove equivariance of $\Phi_{h}^{(1)}$ and $\Phi_{h}^{(2)}$ individually. 
Further, as we have an explicit formula for $\Psi=\left(\Phi_{h}^{(1)}\right)^{-1}$, and also that equivariance of $\Psi$ implies equivariance of $\Phi_{h}^{(1)}$, we will prove equivariance of $\Psi$ with respect to $F=\mathrm{SU}(N)\ltimes\mathfrak{su}(N)^{*}$ action: $\Psi\circ F=F\circ\Psi$.
\begin{equation*}
(\Psi\circ F)(\tilde Q,\tilde P)=(Q_{n},P_{n}),
\end{equation*}
where
\begin{equation*}
Q_{n}=\tilde QG-\frac{h}{2}f\left(\tilde QG,\tilde P(G^{\dagger})^{-1}\right)=\tilde QG+\frac{h}{2}M_{1}\left(W(\tilde QG,\tilde P(G^{-1})^{\dagger})\right)\tilde QG.
\end{equation*}
Further, using \eqref{mommap} we get
\begin{equation*}
\begin{split}
W(\tilde QG,\tilde P(G^{-1})^{\dagger})&=\frac{1}{2}\left(\tilde QGG^{-1}\tilde P^{\dagger}-\tilde P(G^{-1})^{\dagger}G^{\dagger}\tilde Q^{\dagger}\right)={}\\&=\frac{1}{2}\left(\tilde Q\tilde P^{\dagger}-\tilde P\tilde Q^{\dagger}\right)=W(\tilde Q,\tilde P).
\end{split}
\end{equation*}
Therefore, $f\left(\tilde QG,\tilde P(G^{\dagger})^{-1}\right)=-M_{1}(W(\tilde Q,\tilde P))\tilde QG$, and finally
\begin{equation*}
Q_{n}=\left(\tilde Q-\frac{h}{2}f(\tilde Q,\tilde P)\right)G.
\end{equation*}
For $P_{n}$, we get
\begin{equation*}
P_{n}=\tilde P(G^{\dagger})^{-1}-\frac{h}{2}g\left(\tilde QG,\tilde P(G^{\dagger})^{-1},G^{-1}\alpha G\right),
\end{equation*}
with 
\begin{equation*}
g\left(\tilde QG,\tilde P(G^{\dagger})^{-1},G^{-1}\alpha G\right)=M_{1}^{\dagger}\tilde P(G^{\dagger})^{-1}+2M_{2}^{\dagger}\tilde QG(G^{-1}\alpha G)^{\dagger},
\end{equation*}
where $M_{1}=M_{1}\left(W(\tilde QG,\tilde P(G^{-1})^{\dagger}\right)$, and $M_{2}=M_{2}\left(\Theta(\tilde QG,G^{-1}\alpha G)\right)$.

We have already shown that $W\left(\tilde QG,\tilde P(G^{-1})^{\dagger}\right)=W\left(\tilde Q,\tilde P\right)$. Now,
\begin{equation*}
\Theta(\tilde QG,G^{-1}\alpha G)=\tilde QG(G^{-1}\alpha G)G^{-1}\tilde Q^{\dagger}=\tilde Q\alpha\tilde Q^{\dagger}=\Theta(\tilde Q,\alpha).
\end{equation*}
Therefore, we also have that $M_{2}(\Theta(\tilde QG,G^{-1}\alpha G))=M_{2}(\Theta(\tilde Q,\alpha))$. Further,
\begin{equation*}
\begin{split}
g\left(\tilde QG,\tilde P(G^{\dagger})^{-1},G^{-1}\alpha G\right)&=(M_{1}^{\dagger}\tilde P+2M_{2}^{\dagger}\tilde Q\alpha^{\dagger})(G^{-1})^{\dagger}={}\\&=g(\tilde Q,\tilde P,\alpha)(G^{-1})^{\dagger}.
\end{split}
\end{equation*}
Thus, for $P_{n}$ we get
\begin{equation*}
P_{n}=\left(\tilde P-\frac{h}{2}g(\tilde Q,\tilde P,\alpha)\right)(G^{-1})^{\dagger},
\end{equation*}
and finally
\begin{equation}
\label{psiF}
(\Psi\circ F)(\tilde Q,\tilde P)=\left(\left(\tilde Q-\frac{h}{2}f(\tilde Q,\tilde P)\right)G,\left(\tilde P-\frac{h}{2}g(\tilde Q,\tilde P,\alpha)\right)(G^{-1})^{\dagger}\right).
\end{equation}
On the other hand,
\begin{equation}
\label{Fpsi}
\begin{split}
(F\circ\Psi)(\tilde Q,\tilde P)&=F\left(\tilde Q-\frac{h}{2}f(\tilde Q,\tilde P),\tilde P-\frac{h}{2}g(\tilde Q,\tilde P,\alpha)\right)={}\\&=\left(\left(\tilde Q-\frac{h}{2}f(\tilde Q,\tilde P)\right)G,\left(\tilde P-\frac{h}{2}g(\tilde Q,\tilde P,\alpha)\right)(G^{-1})^{\dagger}\right).
\end{split}
\end{equation}
Comparing \eqref{psiF} and \eqref{Fpsi} we conclude that $\Psi=\left(\Phi_{h}^{(1)}\right)^{-1}$ is equivariant, and therefore $\Phi_{h}^{(1)}$ is so as well.

Since maps $\left(\Phi_{h}^{(1)}\right)^{-1}$ and $\Phi_{h}^{(2)}$ differ by a sign in front of $h/2$ (see \eqref{implMP}), the proof of equivariance for $\Phi_{h}^{(2)}$ is exactly the same as for $\left(\Phi_{h}^{(1)}\right)^{-1}$, and $\Phi_{h}=\Phi_{h}^{(2)}\circ\Phi_{h}^{(1)}$ is equivariant.
This concludes the proof.
\end{proof}

We can see in equation \eqref{implMP} that the implicit midpoint method is not formulated intrinsically on $T^{*}\mathrm{SU}(N)$. 
Namely, the matrix $\tilde Q$ does not necessarily belong to $\mathrm{SU}(N)$. 
However, the flow of matrices $(W,\Theta)$ still remains on $\mathfrak{su}(N)\ltimes\mathfrak{su}(N)^{*}$ due to \eqref{mommap}. 
Moreover, the proof of equivariance does not use that $\tilde Q\in\mathrm{SU}(N)$. 

Finally, we arrive at the following result.
\begin{theorem}
The implicit midpoint method \eqref{implMP} for the Hamiltonian system \eqref{hamsys} descends to a Lie--Poisson integrator 
$$
\phi_{h}\colon\mathfrak{f}^{*}\to\mathfrak{f}^{*}, \quad (W_{n},\Theta_{n})\mapsto(W_{n+1},\Theta_{n+1})
$$
for the Lie--Poisson flow \eqref{qEu}. 
The method is defined by the following equations:
\begin{equation}
\label{MHDmethod1}
\begin{aligned}
&\Theta_{n}=\tilde\Theta-\frac{h}{2}[\tilde\Theta,\tilde{M_{1}}]-\frac{h^{2}}{4}\tilde{M_{1}}\tilde\Theta\tilde{M_{1}}, \\
&\Theta_{n+1}=\Theta_{n}+h[\tilde\Theta,\tilde{M_{1}}],\\
&W_{n}=\tilde W-\frac{h}{2}[\tilde W,\tilde{M_{1}}]-\frac{h}{2}[\tilde\Theta,\tilde M_{2}]-\frac{h^{2}}{4}\left(\tilde{M_{1}}\tilde W\tilde{M_{1}}+\tilde M_{2}\tilde\Theta\tilde M_{1}+\tilde M_{1}\tilde\Theta\tilde M_{2}\right), \\
&W_{n+1}=W_{n}+h[\tilde W,\tilde{M_{1}}]+h[\tilde\Theta,\tilde{M_{2}}],
\end{aligned}
\end{equation}
where $\tilde{M_{1}}=\Delta_{N}^{-1}(\tilde W)$, $\tilde{M_{2}}=\Delta_{N}(\tilde\Theta)$.
Furthermore, this integrator preserves the Casimirs \eqref{qCas}:
\begin{equation*}
\begin{aligned}
\mathrm{tr}(f(\Theta_{n}))&=\mathrm{tr}(f(\Theta_{n+1})),\\
\mathrm{tr}(W_{n}g(\Theta_{n}))&=\mathrm{tr}(W_{n+1}g(\Theta_{n+1})).
\end{aligned}
\end{equation*}
\end{theorem}
\begin{proof}
The formulae \eqref{MHDmethod1} are obtained straightforwardly by means of
\begin{equation*}
\begin{aligned}
W_{n}&=\frac{1}{2}\left(Q_{n}P_{n}^{\dagger}-P_{n}Q_{n}^{\dagger}\right),\\
\Theta_{n}&=Q_{n}\alpha Q_{n}^{\dagger},\\
W_{n+1}&=\frac{1}{2}\left(Q_{n+1}P_{n+1}^{\dagger}-P_{n+1}Q_{n+1}^{\dagger}\right),\\
\Theta_{n+1}&=Q_{n+1}\alpha Q_{n+1}^{\dagger} .
\end{aligned}
\end{equation*}
Preservation of Casimirs is a direct consequence of Theorem \ref{th1} and Lemma \ref{lemm1}.
\end{proof}

\begin{remark}
The scheme \eqref{MHDmethod1} has order of consistency $O(h^{2})$, the same as the underlying symplectic Runge-Kutta method \eqref{implMP}. 
It is straightforward to generalize to an integrator of arbitrary order $O(h^{s})$, once we apply symplectic $s$-stage Runge-Kutta scheme on $T^{*}F$.
\end{remark}

\subsection{Algebras other than $\mathfrak{su}(N)$}

Here, we show that the above formalism allows us to develop structure preserving integrators also for other Lie algebras defined by different constraints. We start with $J$-quadratic Lie algebras.
\subsubsection{$J$-quadratic Lie algebras}
Let $\mathfrak{g}$ be $J$-quadratic, that is \begin{equation}
\label{Jquadr}
A^{\dagger}J+JA=0
\end{equation}for any $A\in\mathfrak{g}$ and $J$ being $J^{2}=cI$, with $c\in\mathbb{R}\setminus\left\{0\right\}$, and $J=\pm J^{\dagger}$. This setting covers most of the classical Lie algebras, such that $\mathfrak{su}(N)$, $\mathfrak{u}(N)$, $\mathfrak{so}(N)$, $\mathfrak{sp}(N,\mathbb{C})$, $\mathfrak{sp}(N,\mathbb{R})$. 

The relation \eqref{Jquadr} implies the following quadratic constraint on the group $\mathrm{GL}(N,\mathbb{C})$ defining the corresponding matrix Lie group $G$:
\begin{equation*}
Q\in G\Longleftrightarrow Q^{\dagger}JQ=J.
\end{equation*}
Therefore, the momentum map $\mu\colon T^{*}F\to\mathfrak{f}^{*}$ is
\begin{equation}
\label{mommapJquadr}
\begin{split}
\mu(Q,m,P,\alpha)&=\left(\Pi(QP^{\dagger}),\,Q\alpha Q^{-1}\right)={}\\&=\left(\frac{1}{2}\left(QP^{\dagger}-\frac{1}{c}JPQ^{\dagger}J\right),\,\frac{1}{c}Q\alpha JQ^{\dagger}J\right)=(W,\Theta),
\end{split}
\end{equation}
where $\alpha\in\mathfrak{g}$, and \begin{equation*}\Pi\colon\mathfrak{gl}(N,\mathbb{C})\to\mathfrak{g},\quad \Pi(W)=\frac{1}{2}\left(W-\frac{1}{c}JW^{\dagger}J\right),\quad W\in\mathfrak{gl}(N,\mathbb{C}).
\end{equation*}
is a projector onto the Lie algebra $\mathfrak{g}$. 

The Hamilton's equations on $T^{*}F$ are
\begin{equation}
\label{hamsys1}
\left\{
\begin{aligned}
&\dot Q=-M_{1}Q, \\
&\dot P=M_{1}^{\dagger}P+\frac{2}{c}JM_{2}Q\alpha J,\\
&\dot\alpha=0,
\end{aligned}
\right.
\end{equation}
Applying the implicit midpoint method to \eqref{hamsys1} and using \eqref{mommapJquadr}, we get
\begin{align*}
&\Theta_{n}=\frac{1}{c}Q_{n}\alpha JQ_{n}^{\dagger}J=\frac{1}{c}\left(\tilde Q+\frac{h}{2}\tilde{M_{1}}\tilde Q\right)\alpha J\left(\tilde Q+\frac{h}{2}\tilde{M_{1}}\tilde Q\right)^{\dagger}J={}\\& \quad\; =\left(I+\frac{h}{2}\tilde{M_{1}}\right)\tilde\Theta\left(I-\frac{h}{2}\tilde{M_{1}}\right),{}\\
&\Theta_{n+1}=\Theta_{n}+h[\tilde\Theta,\tilde{M_{1}}], \\
&W_{n}=\tilde W-\frac{h}{2}[\tilde W,\tilde{M_{1}}]-\frac{h}{2}[\tilde\Theta,\tilde M_{2}]-\frac{h^{2}}{4}\left(\tilde{M_{1}}\tilde W\tilde{M_{1}}+\tilde M_{2}\tilde\Theta\tilde M_{1}+\tilde M_{1}\tilde\Theta\tilde M_{2}\right),\\
&W_{n+1}=W_{n}+h[\tilde W,\tilde{M_{1}}]+h[\tilde\Theta,\tilde{M_{2}}],
\end{align*}
which coincides with \eqref{MHDmethod1}.

\begin{remark}
Since the method just derived has the same form for all $J$-quadratic Lie algebras, we have thus extended the previous setting from $\mathfrak{f}=\mathfrak{su}(N)\ltimes\mathfrak{su}(N)^{*}$ to  $\mathfrak{f}=\mathfrak{g}\ltimes\mathfrak{g}^{*}$ for an arbitrary $J$-quadratic Lie algebra $\mathfrak{g}$.
In this way, the method \eqref{MHDmethod1} represents the natural extension of \textit{isospectral symplectic Runge-Kutta methods} \cite{ModViv} for Lie--Poisson systems on $J$-quadratic Lie algebras $\mathfrak{g}$ to those on the \textit{magnetic extension} $\mathfrak{f}=\mathfrak{g}\ltimes\mathfrak{g}^{*}$ of $\mathfrak{g}$. 
We therefore call these integrators \eqref{MHDmethod1} \textit{magnetic symplectic Runge-Kutta methods}.
\end{remark}

\subsubsection{General type Lie algebras}
Here, we do not assume that $\mathfrak{g}$ is a $J$-quadratic Lie algebra, in other words the Lie group $G$ does not allow for the constraint $Q^{\dagger}JQ=J$. In this case one has to use the general formula for the momentum map $\mu\colon T^{*}F\to\mathfrak{f}^{*}$:
\begin{equation*}
\mu(Q,m,P,\alpha)=\left(\Pi(P^{\dagger}Q),\,Q\alpha Q^{-1}\right)=(W^{\dagger},\Theta).
\end{equation*}
The Hamiltonian equations are the same as \eqref{hamsys1}, and the integration scheme, in particular, for the $\Theta$ variable is
\begin{equation*}
\begin{aligned}
&\Theta_{n}=\left(I+\frac{h}{2}\tilde{M_{1}}\right)\tilde\Theta\left(I+\frac{h}{2}\tilde{M_{1}}\right)^{-1},\\
&\Theta_{n+1}=\left(I-\frac{h}{2}\tilde{M_{1}}\right)\tilde\Theta\left(I-\frac{h}{2}\tilde{M_{1}}\right)^{-1}.
\end{aligned}
\end{equation*}
One can see that in this case there is no way to get rid of inverse matrix operations, and therefore  discrete semi-direct product reduction theory provides us with a Lie--Poisson integrator different from \eqref{MHDmethod1}.
However, the question if there are mathematically reasonable and practically important examples of such Lie--Poisson flows remains open.

\section{Hazeltine's equations for magnetized plasma}\label{sec:hazeltine}
In this section, we consider another important example of a Lie--Poisson system on the dual of a semi-direct product Lie algebra, which is Hazeltine's equations, describing 2D turbulence in magnetized plasma \cite{Holm,Haz,ShukYuRah,HazHolm,HazMeiss}. 
This system is a generalization of the MHD equations \eqref{MHDvort} considered previously, namely
\begin{equation}
\label{Alfven}
\left\{
\begin{aligned}
&\dot\omega=\left\{\omega,\Delta^{-1}\omega\right\}+\left\{\theta,\Delta\theta\right\}, \\
&\dot\theta=\left\{\theta,\Delta^{-1}\omega\right\}-\alpha\left\{\theta,\chi\right\},\\
&\dot\chi=\left\{\chi,\Delta^{-1}\omega\right\}+\left\{\theta,\Delta\theta\right\},
\end{aligned}
\right.
\end{equation}
where $\omega$ and $\theta$ have the same meaning as before, $\chi$ is the normalized deviation of the particle density from a constant equilibrium value, and $\alpha$ is a constant parameter. 
If $\alpha=0$, the system \eqref{Alfven} decouples into the dynamics of the two fields $\omega$ and $\theta$ that constitutes the MHD dynamics, and the dynamics of the $\chi$ field.

The system \eqref{Alfven} is known to be a Lie--Poisson system, as first described in \cite{HazHolm}, for the Hamiltonian
\begin{equation*}
H=\frac{1}{2}\int\limits_{S^{2}}\left(\omega\Delta^{-1}\omega+\theta\Delta\theta-\alpha\chi^{2}\right)\mu,
\end{equation*}
and with the Casimirs
\begin{equation*}
\mathcal{C}=\int\limits_{S^{2}}\left(f(\theta)+\chi g(\theta)+k(\omega-\chi)\right)\mu
\end{equation*}
for arbitrary smooth functions $f,g,k$.

Using the geometric quantization approach as described above, we get a spatially discretized analogue of the system \eqref{Alfven} given by
\begin{equation}
\label{qAlfven}
\left\{
\begin{aligned}
&\dot W=[W,M_{1}]+[\Theta,M_{2}], \\
&\dot\Theta=[\Theta,M_{1}]-\alpha[\Theta,\chi],\\
&\dot\chi=[\chi,M_{1}]+[\Theta,M_{2}],
\end{aligned}
\right.
\end{equation}
where $W,\Theta,\chi\in\mathfrak{su}(N)$, and $M_{1}=\Delta_{N}^{-1}W$, $M_{2}=\Delta_{N}\Theta$.

By introducing a new variable $\Psi=W-\chi$, the system \eqref{qAlfven} becomes
\begin{equation}
\label{qAlfven1}
\left\{
\begin{aligned}
&\dot \Psi=[\Psi,M_{1}], \\
&\dot\Theta=[\Theta,M_{3}],\\
&\dot\chi=[\chi,M_{3}]+[\Theta,M_{2}],
\end{aligned}
\right.
\end{equation}
where $M_{3}=M_{1}-\alpha\chi$.

\begin{proposition}
The system \eqref{qAlfven1} is a Lie--Poisson flow on the dual $\mathfrak{f}^{*}$ of the Lie algebra $$\mathfrak{f}=\mathfrak{su}(N)\oplus\left(\mathfrak{su}(N)\ltimes\mathfrak{su}(N)^{*}\right).$$
\end{proposition}

The quantized analogues of the Casimirs for \eqref{Alfven} are:
\begin{itemize}
\item the spectrum of $\Psi=W-\chi$, or equivalently
\begin{equation*}
\mathcal{E}_{k}=\mathrm{tr}(k(W-\chi))
\end{equation*}
for any smooth function $k$;
\item the spectrum of $\Theta$, or equivalently
\begin{equation*}
\mathcal{C}_{f}=\mathrm{tr}(f(\Theta))
\end{equation*}
for any smooth function $f$;
\item the cross-helicity
\begin{equation*}
J=\mathrm{tr(\chi g(\Theta))}
\end{equation*}
for any smooth function $g$.
\end{itemize}
We also have the Hamiltonian as a conserved quantity:
\begin{equation}
\label{hamAlfven}
H=\frac{1}{2}\mathrm{tr}\left(WM_{1}+\Theta M_{2}-\alpha\chi^{2}\right).
\end{equation}

As we see from \eqref{qAlfven1}, we can apply the isospectral (midpoint) integrator \cite{ModViv} to the first equation for $\Psi$, and the magnetic (midpoint) integrator \eqref{MHDmethod1} to the pair of equations for $\Theta$ and $\chi$. This results in the following scheme for the variables $W$, $\Theta$, and $\chi$:
\begin{equation}
\label{IsoAlfven}
\begin{aligned}
&\Theta_{n}=\tilde\Theta-\frac{h}{2}[\tilde\Theta,\tilde{M_{3}}]-\frac{h^{2}}{4}\tilde{M_{3}}\tilde\Theta\tilde{M_{3}}, \\
&\Theta_{n+1}=\Theta_{n}+h[\tilde\Theta,\tilde{M_{3}}],\\
&W_{n}=\tilde W-\frac{h}{2}[\tilde W,\tilde{M_{1}}]-\frac{h}{2}[\tilde\Theta,\tilde M_{2}]-\\&\frac{h^{2}}{4}\left(\tilde{M_{1}}\tilde W\tilde{M_{1}}+\tilde M_{2}\tilde\Theta\tilde M_{3}+\tilde M_{3}\tilde\Theta\tilde M_{2}-\alpha\tilde{M_{1}}\tilde{\chi}^{2}-\alpha\tilde{\chi}^{2}\tilde{M_{1}}+\alpha^{2}\tilde{\chi}^{3}\right), \\
&W_{n+1}=W_{n}+h[\tilde W,\tilde{M_{1}}]+h[\tilde\Theta,\tilde{M_{2}}],\\
&\chi_{n}=\tilde\chi-\frac{h}{2}[\tilde\chi,\tilde{M_{3}}]-\frac{h}{2}[\tilde\Theta,\tilde M_{2}]-\frac{h^{2}}{4}\left(\tilde{M_{3}}\tilde \chi\tilde{M_{3}}+\tilde M_{2}\tilde\Theta\tilde M_{3}+\tilde M_{3}\tilde\Theta\tilde M_{2}\right),\\
&\chi_{n+1}=\chi_{n}+h[\tilde\chi,\tilde{M_{3}}]+h[\tilde\Theta,\tilde{M_{2}}],
\end{aligned}
\end{equation}
where $\tilde{M_{1}}=\Delta_{N}^{-1}(\tilde W)$, $\tilde{M_{2}}=\Delta_{N}(\tilde\Theta)$, $\tilde{M_{3}}=\tilde{M_{1}}-\alpha\tilde\chi$.

\begin{proposition}
\label{thm1}
The numerical scheme \eqref{IsoAlfven} is a Lie--Poisson integrator for \eqref{qAlfven}. 
It preserves the Casimirs exactly,
\begin{equation*}
\begin{aligned}
&\mathrm{tr}(k(W_{n}-\chi_{n}))=\mathrm{tr}(k(W_{n+1}-\chi_{n+1})),\\
&\mathrm{tr}(f(\Theta_{n}))=\mathrm{tr}(f(\Theta_{n+1})),\\
&\mathrm{tr(\chi_{n} g(\Theta_{n}))}=\mathrm{tr(\chi_{n+1}g(\Theta_{n+1}))},
\end{aligned}
\end{equation*}
and nearly preserves the Hamiltonian \eqref{hamAlfven} in the sense of backward error analysis.
\end{proposition}

\section{Kirchhoff equations}\label{sec:kirchhoff_equations}
Another example of a Lie--Poisson system on the dual of a semi-direct product Lie algebra is the \textit{Kirchhoff equations}, describing the motion of a rigid body in an ideal fluid. 
This is a Lie--Poisson flow on the dual of $\mathfrak{f}=\mathfrak{so}(3)\ltimes\mathbb{R}^{3}\simeq\mathfrak{so}(3)\ltimes\mathfrak{so}(3)^{*}$, and is thus a magnetic extension of the rigid body dynamics:
\begin{equation}
\label{kir}
\left\{
\begin{aligned}
&\dot m=m\times\omega+p\times u,\\
&\dot p=p\times\omega,
\end{aligned}
\right.
\end{equation}
where $p,m,\omega,u\in\mathbb{R}^{3}$, and
\begin{equation*}
u_{i}=\frac{\partial H}{\partial p_{i}},\quad\omega_{i}=\frac{\partial H}{\partial m_{i}},\quad i=1,2,3,
\end{equation*}
for the Hamiltonian
\begin{equation*}
H(m,p)=\frac{1}{2}\left(\sum\limits_{k=1}^{3}a_{k}m_{k}^{2}+\sum\limits_{k,j=1}^{3}b_{kj}(p_{k}m_{j}+m_{k}p_{j})+\sum\limits_{k,j=1}^{3}c_{kj}p_{k}p_{j}\right),
\end{equation*}
where $a_{k},b_{kj},c_{kj}$ are real numbers.

Using the standard isomorphism between $\mathbb{R}^{3}$ and $\mathfrak{so}(3)$, we construct skew-symmetric matrices
\begin{equation*}
\begin{aligned}
&W=\begin{pmatrix}0 & -m_{3} & m_{2}\\
                  m_{3} & 0 & -m_{1}\\
                  -m_{2} & m_{1} & 0\end{pmatrix},\quad \Theta=\begin{pmatrix}0 & -p_{3} & p_{2}\\
                  p_{3} & 0 & -p_{1}\\
                  -p_{2} & p_{1} & 0\end{pmatrix},\\
&M_{1}=\begin{pmatrix}0 & -\omega_{3} & \omega_{2}\\
                  \omega_{3} & 0 & -\omega_{1}\\
                  -\omega_{2} & \omega_{1} & 0\end{pmatrix},\quad M_{2}=\begin{pmatrix}0 & -u_{3} & u_{2}\\
                  u_{3} & 0 & -u_{1}\\
                  -u_{2} & u_{1} & 0\end{pmatrix}.
\end{aligned}
\end{equation*}
Then, system \eqref{kir} takes the form of a Lie--Poisson flow on the dual of $\mathfrak{f}=\mathfrak{so}(3)\ltimes\mathfrak{so}(3)^{*}$:
\begin{equation*}
\dot W=[W,M_{1}]+[\Theta,M_{2}], \quad \dot\Theta=[\Theta,M_{1}].
\end{equation*}

The Casimirs \eqref{qCas} for the case of Kirchhoff equations are a generalization of the well-known Casimirs
\begin{equation*}
I_{1}=p_{1}^{2}+p_{2}^{2}+p_{3}^{2},\quad I_{2}=m_{1}p_{1}+m_{2}p_{2}+m_{3}p_{3},
\end{equation*}
which are obtained from \eqref{qCas} by taking $f(\Theta)=\Theta^{2}$, and $g(\Theta)=\Theta$.

The following classical cases are known to be integrable for the Kirchhoff equations (for all of them $b_{kj}=c_{kj}=0$ for $j\ne k$):
\begin{itemize}
\item Kirchhoff case \cite{Kir}
\begin{equation*}
a_{1}=a_{2},\quad b_{11}=b_{22},\quad c_{11}=c_{22}.
\end{equation*}
\item Clebsch case \cite{Cleb}
\begin{equation*}
b_{11}=b_{22}=b_{33},\quad \frac{c_{11}-c_{22}}{a_{3}}+\frac{c_{33}-c_{11}}{a_{2}}+\frac{c_{22}-c_{33}}{a_{1}}=0.
\end{equation*}
\item Lyapunov--Steklov--Kolosov case \cite{LSK}
\begin{equation*}
\begin{split}
&\frac{b_{11}-b_{22}}{a_{3}}+\frac{b_{33}-b_{11}}{a_{2}}+\frac{b_{22}-b_{33}}{a_{1}}=0,{}\\&c_{11}-\frac{(b_{22}-b_{33})^{2}}{a_{1}}=c_{22}-\frac{(b_{33}-b_{11})^{2}}{a_{2}}=c_{33}-\frac{(b_{11}-b_{22})^{2}}{a_{3}}.
\end{split}
\end{equation*}
\end{itemize}
In the forthcoming section, we will illustrate the method \eqref{MHDmethod1} by both verifying the preservation of Casimirs, Hamiltonian, and capturing integrable behaviour.
\section{Numerical simulations}
\label{sec:numerical_simulations}
In this section, we provide numerical tests of the schemes \eqref{MHDmethod1} and \eqref{IsoAlfven}, verifying the exact preservation of Casimirs and near preservation of the Hamiltonian.

\subsection{Kirchhoff equations}
We begin verifying the properties of the method on low-dimensional integrable cases of the Kirchhoff equations. 
For all the simulations, we used the time step size $h=0.1$, and the final time of simulation is $T=1000$. 
Initial conditions are randomly generated $\mathfrak{so}(3)$ matrices.\footnote{Numerical simulations for this section are implemented in a Python code available at \url{https://github.com/michaelroop96/kirchhoff.git}}

First, we consider the Kirchhoff integrable case. Fig.~\ref{cas_kir-kir} shows the exact preservation of Casimir functions for the Kirchhoff integrable case, and Fig.~\ref{ham_kir_kir} shows nearly preservation of the Hamiltonian function. 
The phase portrait is shown in Fig.~\ref{phase_kir-kir}. 
One can clearly observe a quasi-periodic dynamics with a regular pattern typical for integrable systems.

\begin{figure}[h!]
\begin{minipage}[h]{0.5\linewidth}
\center{\includegraphics[scale=0.4]{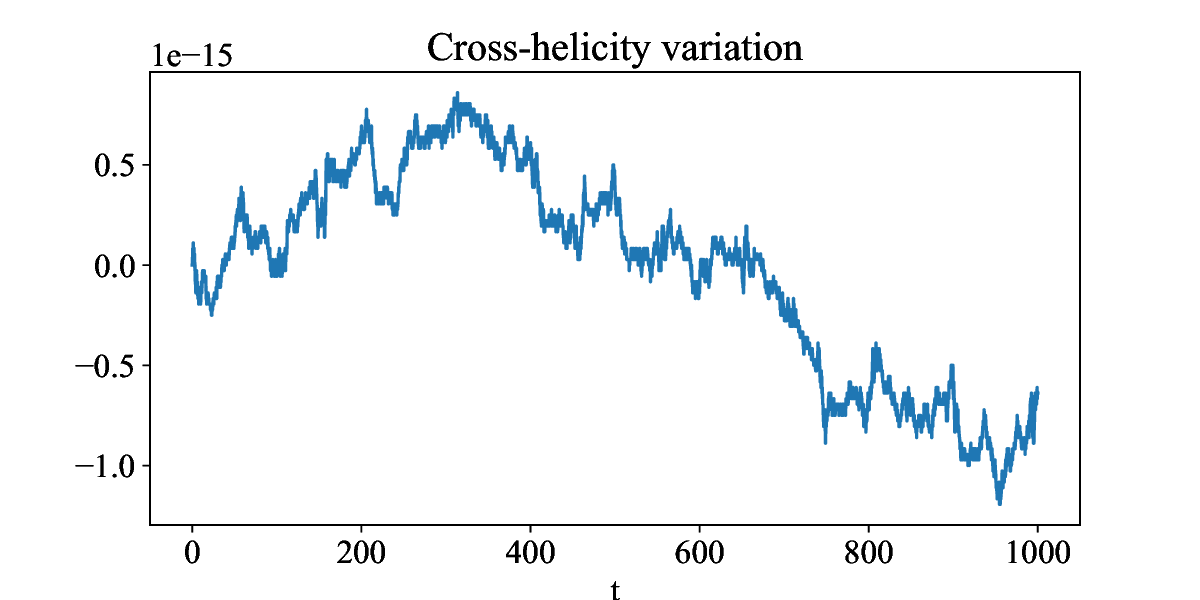}}
\end{minipage}%
\hfill
\begin{minipage}[h]{0.5\linewidth}
\center{\includegraphics[scale=0.4]{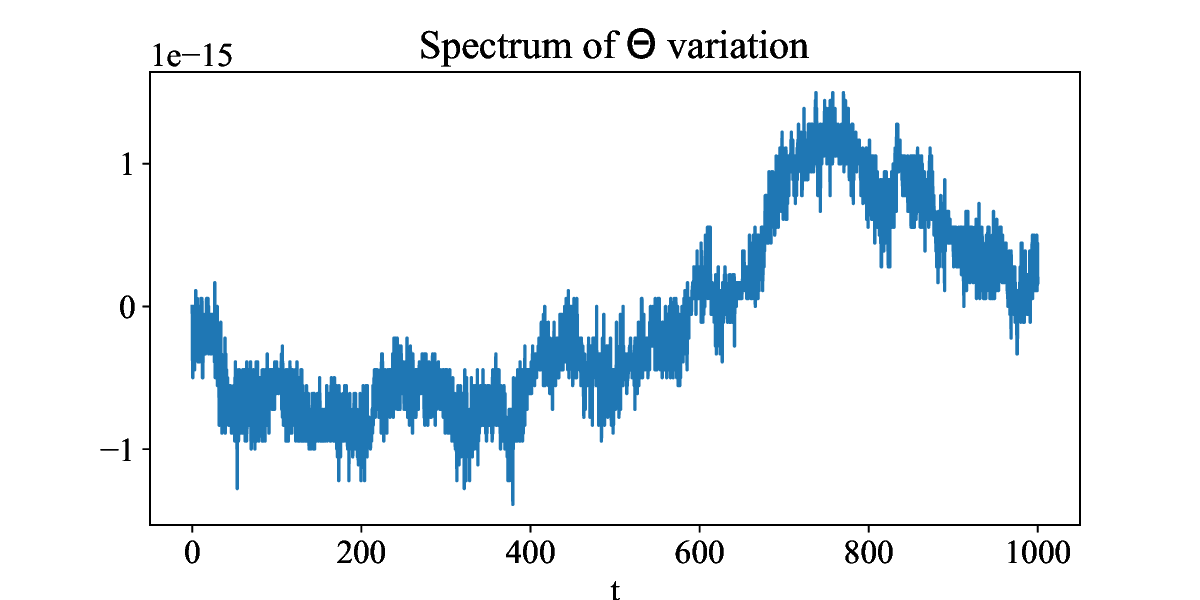}}
\end{minipage}
\hfill
\caption{Cross-helicity (left) and spectrum (right) variation for the Kirchhoff system, Kirchhoff integrable case. The order $10^{-15}$ of the magnitude of the variation indicates the exact preservation of the Casimirs.}
\label{cas_kir-kir}
\end{figure}
\begin{figure}[ht!]
\centering
\includegraphics[scale=0.45]{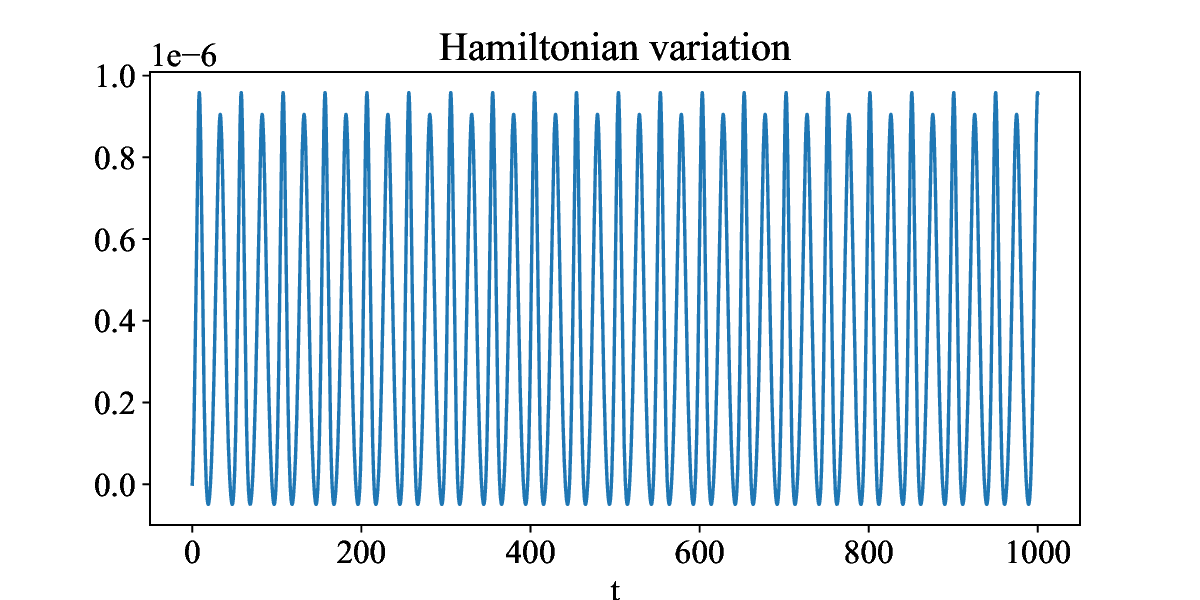}
\caption{Hamiltonian variation for the Kirchhoff system, Kirchhoff integrable case.}
\label{ham_kir_kir}
\end{figure}
\begin{figure}[h]
\begin{minipage}[h]{0.5\linewidth}
\center{\includegraphics[scale=0.35]{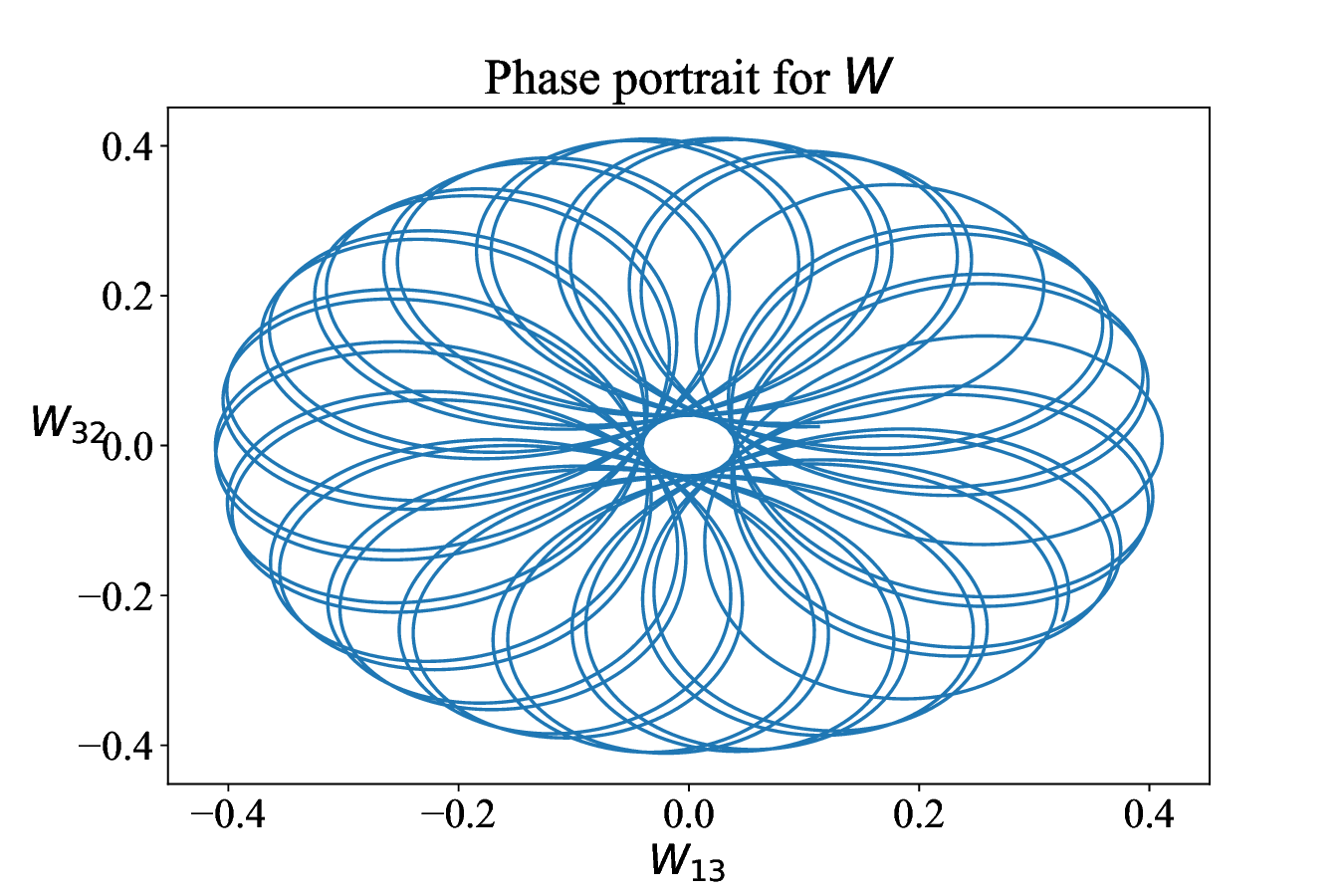}}
\end{minipage}%
\hfill
\begin{minipage}[h]{0.5\linewidth}
\center{\includegraphics[scale=0.35]{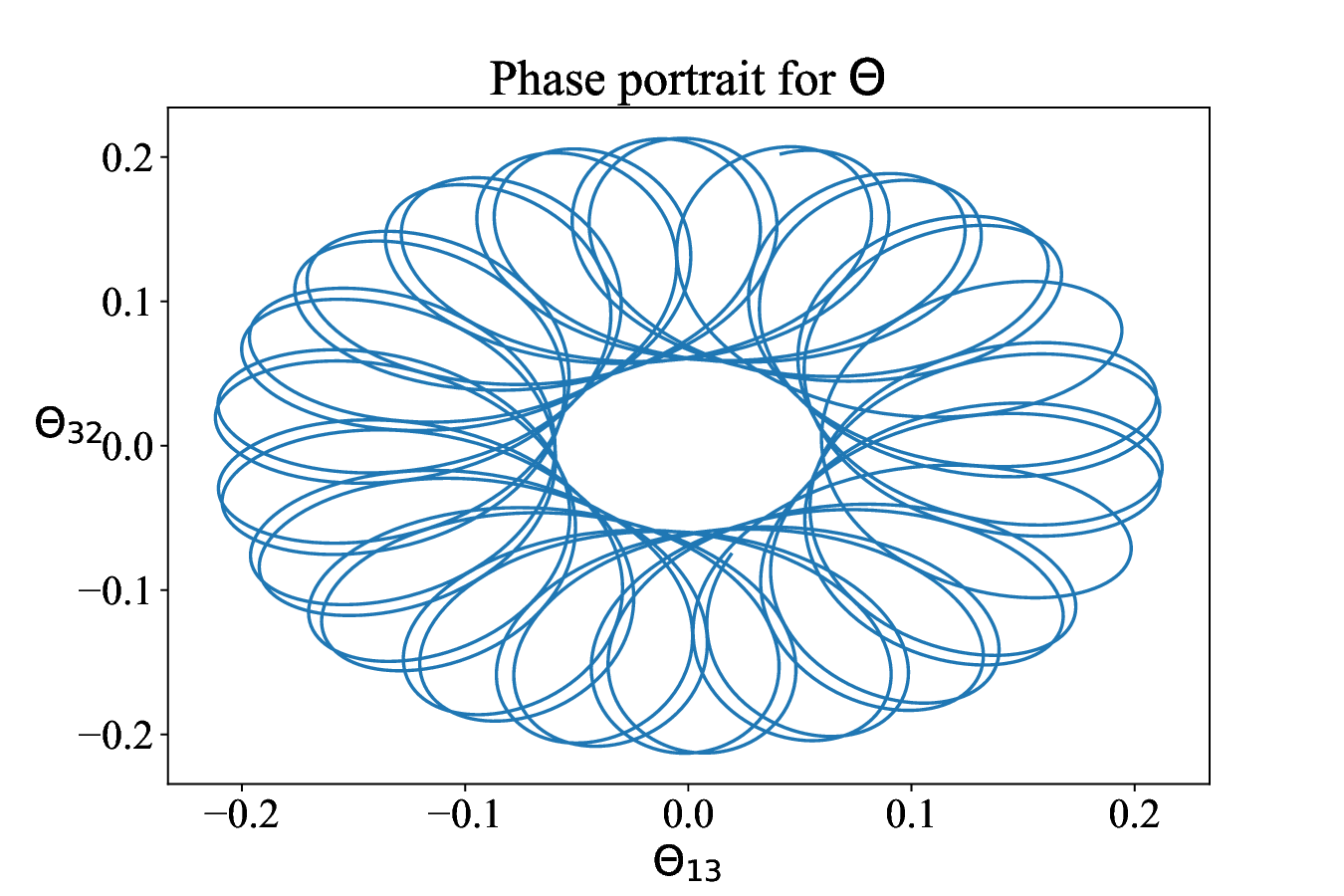}}
\end{minipage}
\hfill
\caption{Phase portrait for $W$ (left) and $\Theta$ (right) for the Kirchhoff system, Kirchhoff integrable case. Component $W_{32}=m_{1}$ is shown against $W_{13}=m_{2}$, and component $\Theta_{32}=p_{1}$ is shown against $\Theta_{13}=p_{2}$.}
\label{phase_kir-kir}
\end{figure}

Second, we consider the Clebsch integrable case, for which we also observe exact preservation of Casimirs in Fig.~\ref{cas_kir-cleb}, nearly preservation of Hamiltonian in Fig.~\ref{ham_kir_cleb}, and a phase portrait in Fig.~\ref{phase_kir_cleb}.

\begin{figure}[h]
\begin{minipage}[h]{0.5\linewidth}
\center{\includegraphics[scale=0.4]{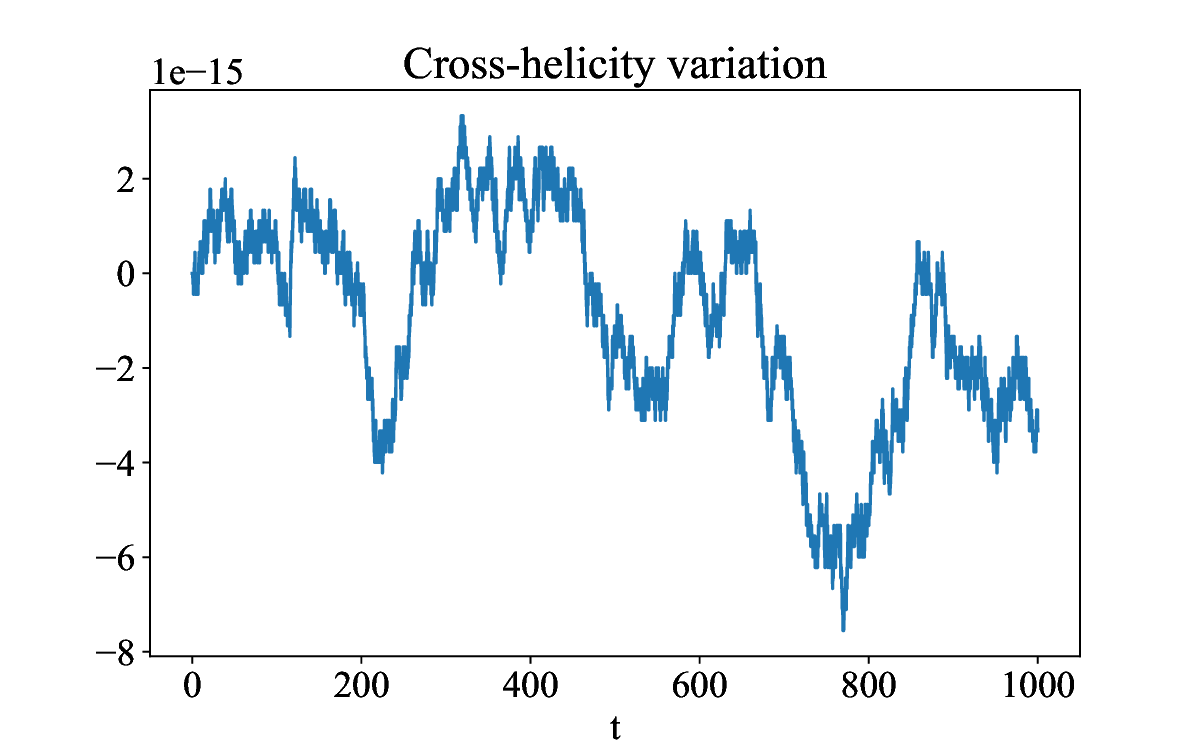}}
\end{minipage}%
\hfill
\begin{minipage}[h]{0.5\linewidth}
\center{\includegraphics[scale=0.4]{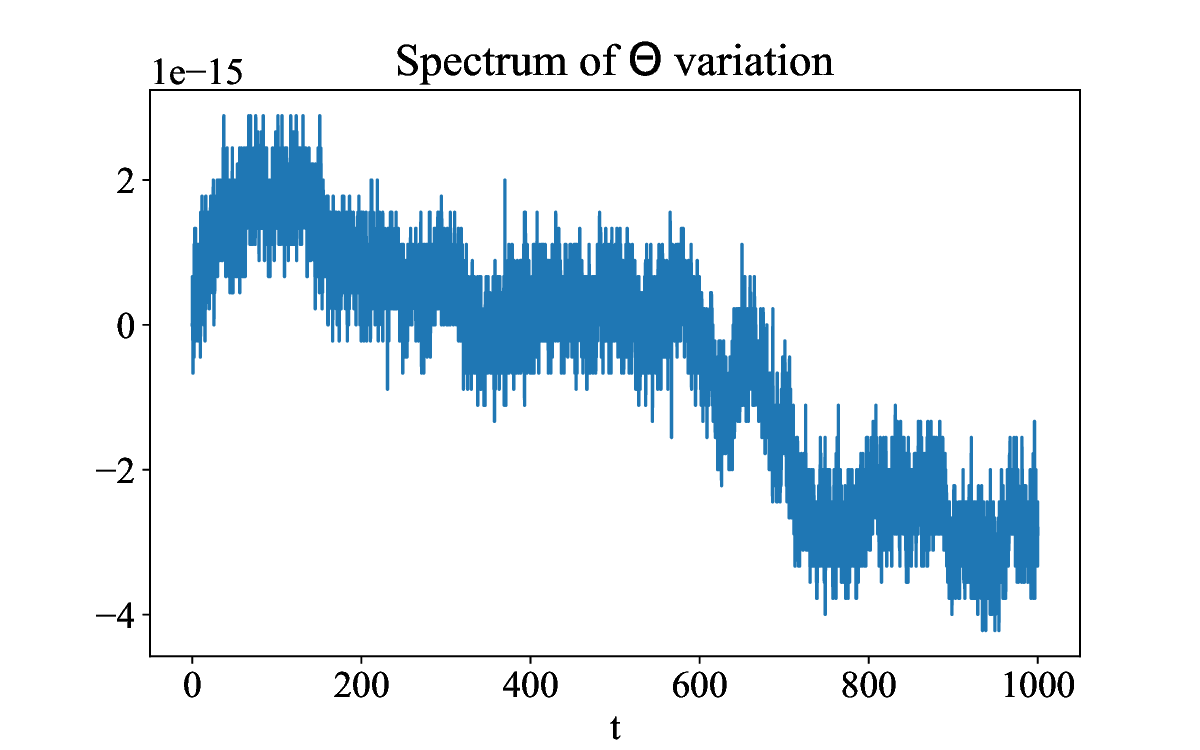}}
\end{minipage}
\hfill
\caption{Cross-helicity (left) and spectrum (right) variation for the Kirchhoff system, Clebsch integrable case. The order $10^{-15}$ of the magnitude of the variation indicates the exact preservation of the Casimirs.}
\label{cas_kir-cleb}
\end{figure}
\begin{figure}[ht!]
\centering
\includegraphics[scale=0.45]{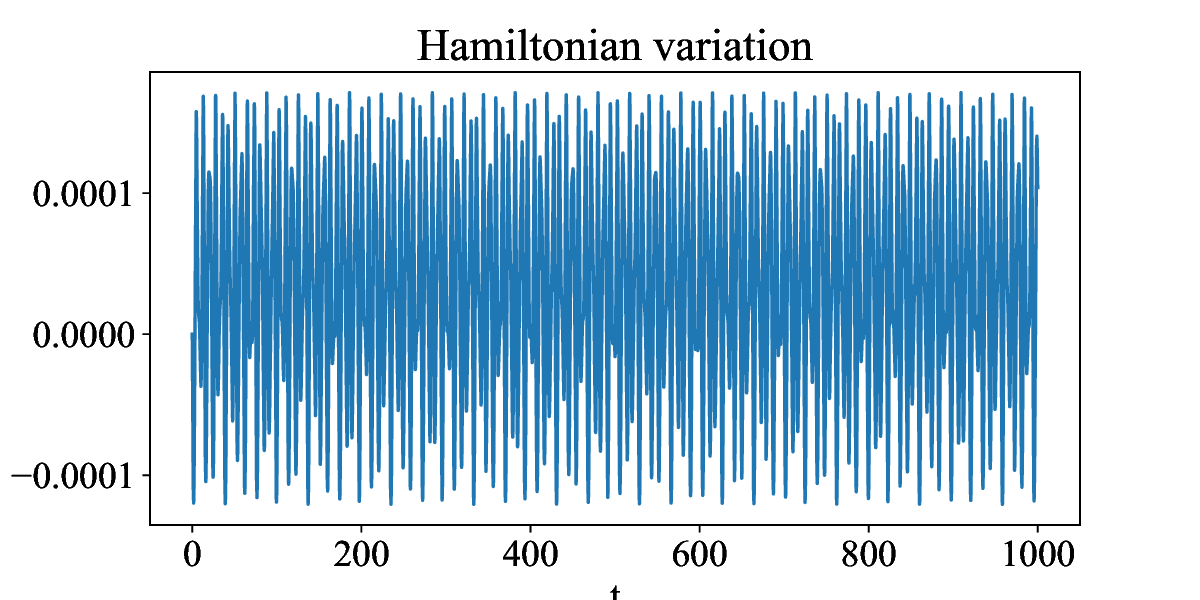}
\caption{Hamiltonian variation for the Kirchhoff system, Clebsch integrable case.}
\label{ham_kir_cleb}
\end{figure}
\begin{figure}[h]
\begin{minipage}[h]{0.5\linewidth}
\center{\includegraphics[scale=0.35]{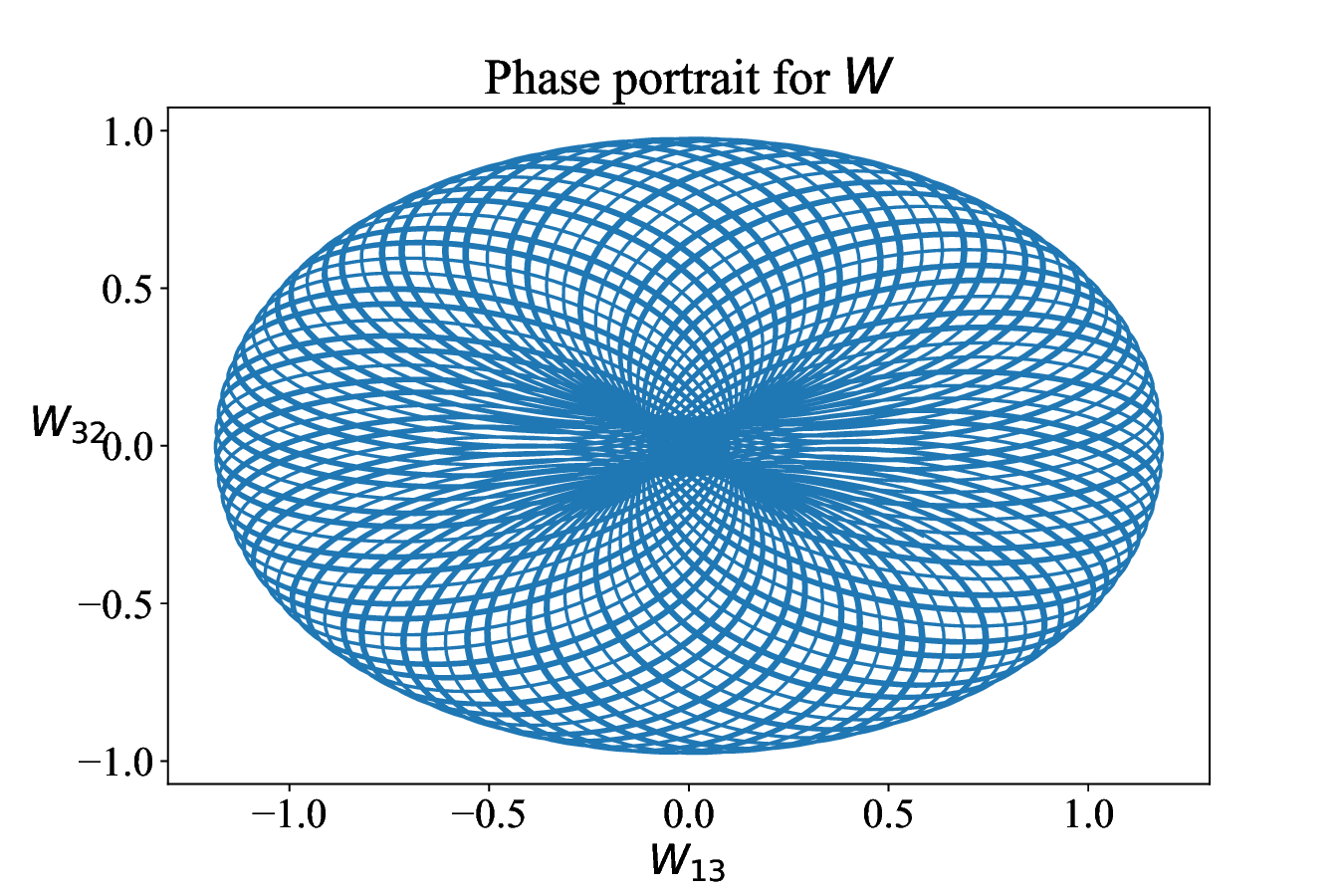}}
\end{minipage}%
\hfill
\begin{minipage}[h]{0.5\linewidth}
\center{\includegraphics[scale=0.35]{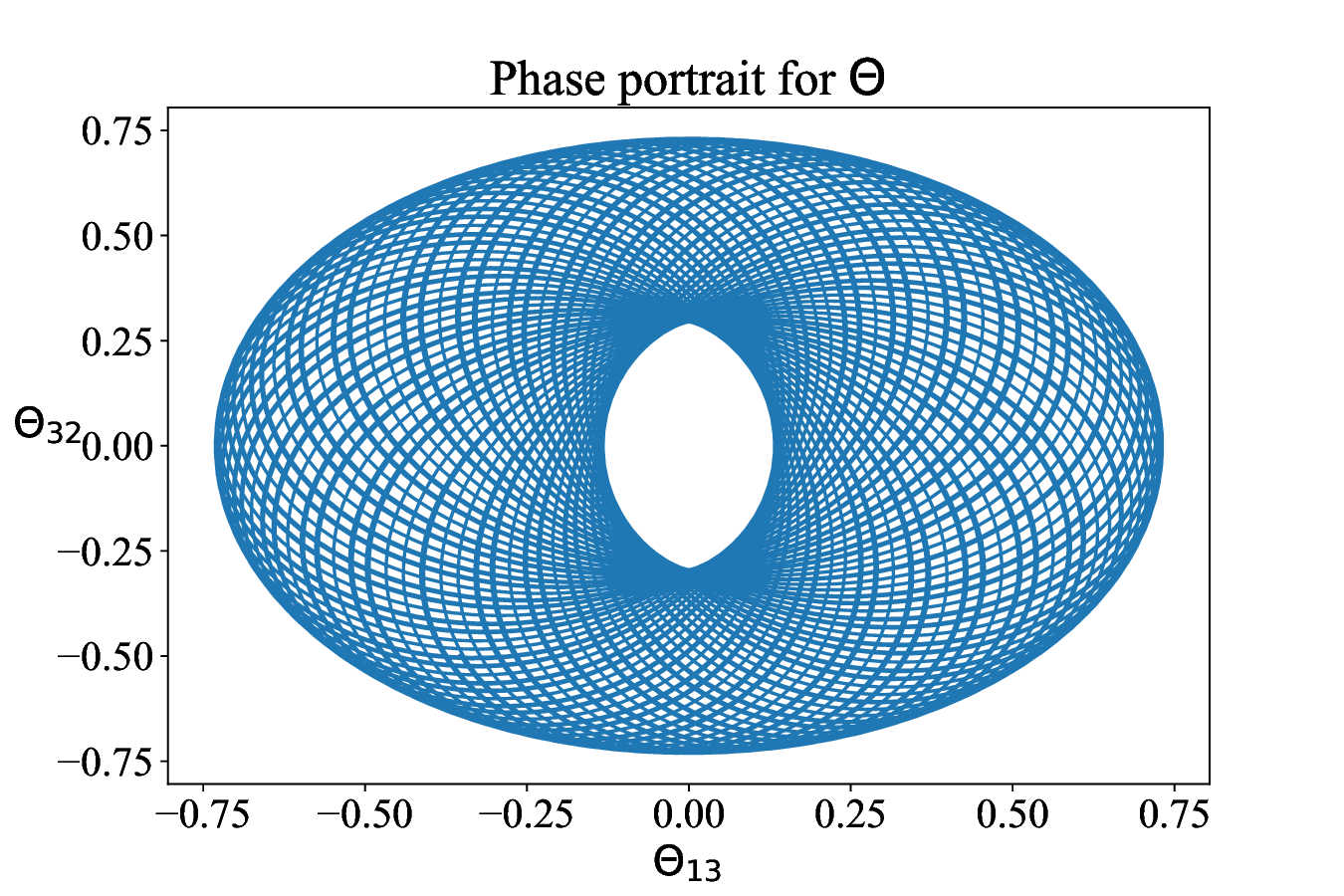}}
\end{minipage}
\hfill
\caption{Phase portrait for $W$ (left) and $\Theta$ (right) for the Kirchhoff system, Clebsch integrable case. Component $W_{32}=m_{1}$ is shown against $W_{13}=m_{2}$, and component $\Theta_{32}=p_{1}$ is shown against $\Theta_{13}=p_{2}$.}
\label{phase_kir_cleb}
\end{figure}

Finally, we have the Lyapunov-Steklov-Kolosov integrable case, with exact preservation of Casimirs shown in Fig.~\ref{cas_kir-LSK}, nearly preservation of Hamiltonian shown in Fig.~\ref{ham_kir_LSK}, and a phase portrait in Fig.~\ref{phase_kir_LSK}.

\begin{figure}[h]
\begin{minipage}[h]{0.5\linewidth}
\center{\includegraphics[scale=0.4]{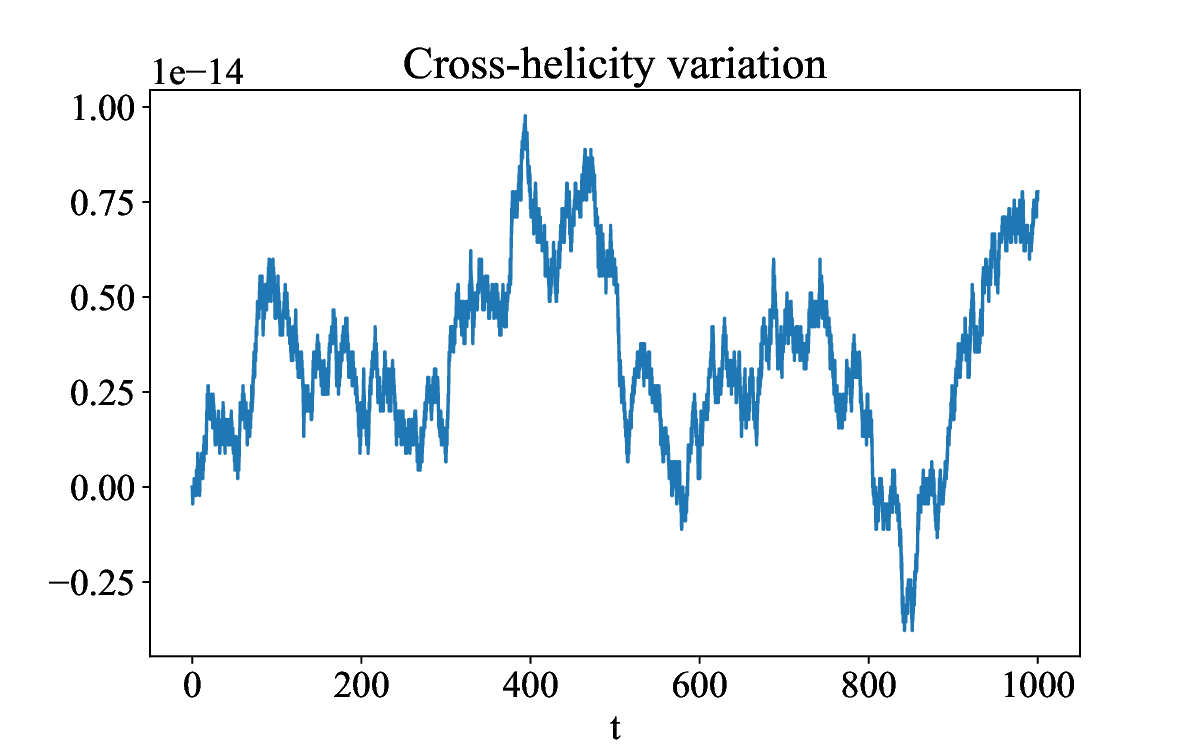}}
\end{minipage}%
\hfill
\begin{minipage}[h]{0.5\linewidth}
\center{\includegraphics[scale=0.4]{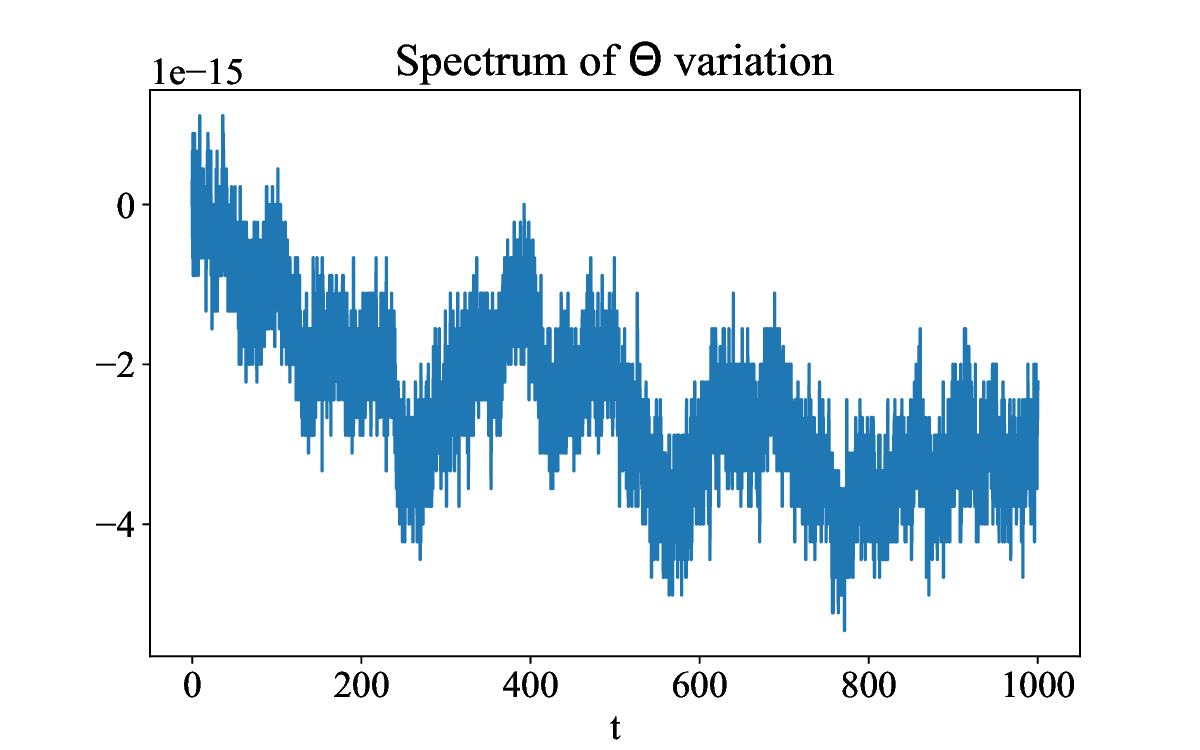}}
\end{minipage}
\hfill
\caption{Cross-helicity (left) and spectrum (right) variation for the Kirchhoff system, Lyapunov-Steklov-Kolosov integrable case. The order $10^{-15}$ of the magnitude of the variation indicates the exact preservation of the Casimirs.}
\label{cas_kir-LSK}
\end{figure}
\begin{figure}[ht!]
\centering
\includegraphics[scale=0.45]{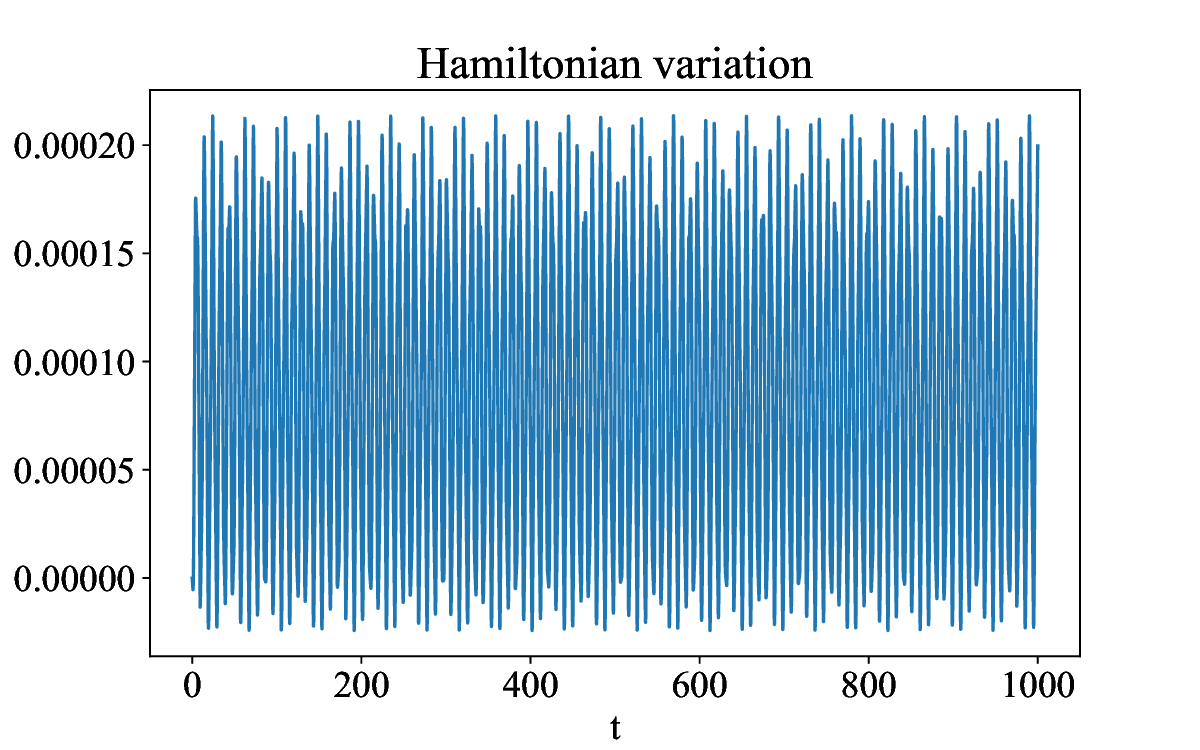}
\caption{Hamiltonian variation for the Kirchhoff system, Lyapunov-Steklov-Kolosov integrable case.}
\label{ham_kir_LSK}
\end{figure}
\begin{figure}[h]
\begin{minipage}[h]{0.5\linewidth}
\center{\includegraphics[scale=0.35]{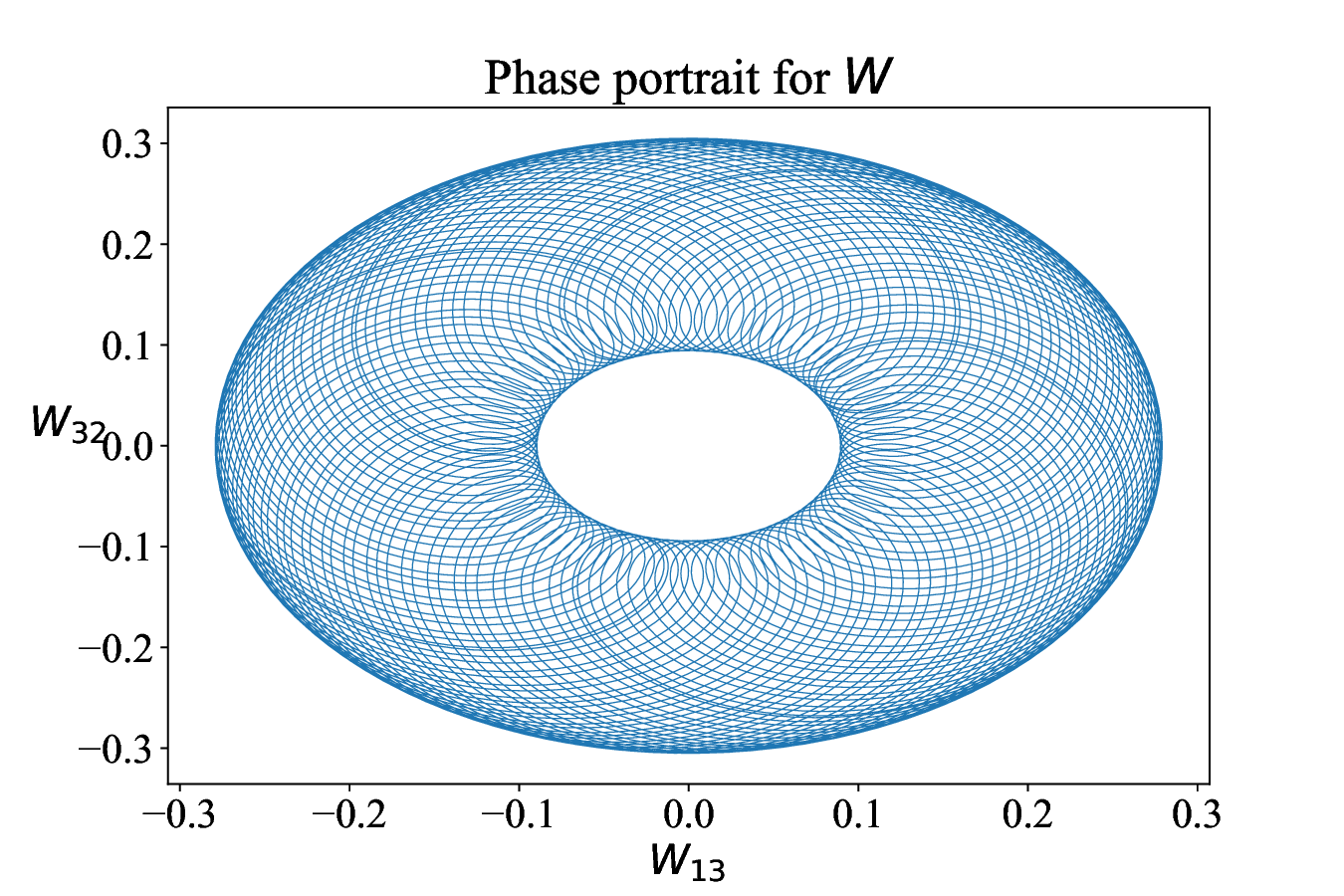}}
\end{minipage}%
\hfill
\begin{minipage}[h]{0.5\linewidth}
\center{\includegraphics[scale=0.35]{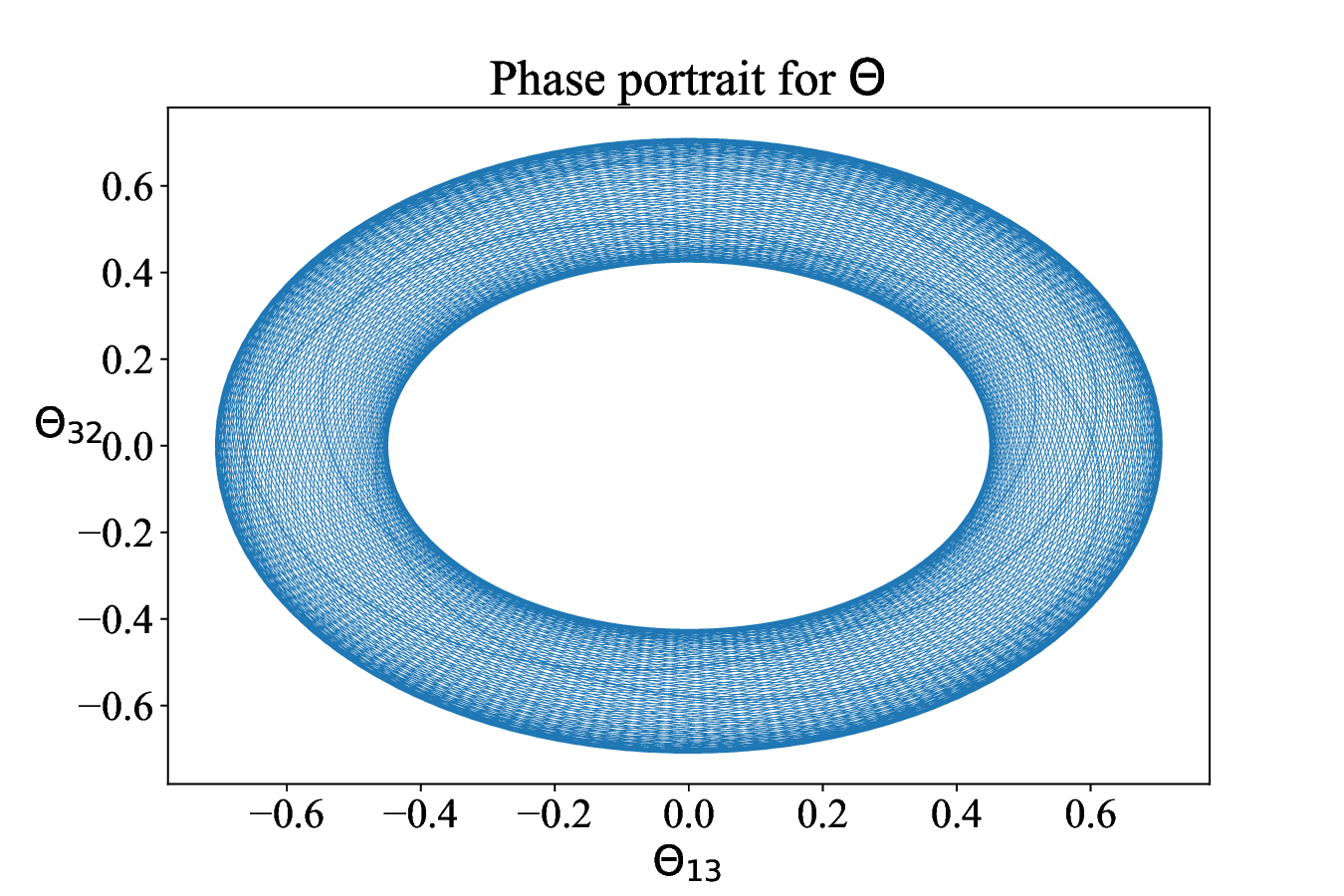}}
\end{minipage}
\hfill
\caption{Phase portrait for $W$ (left) and $\Theta$ (right) for the Kirchhoff system, Lyapunov-Steklov-Kolosov integrable case. Component $W_{32}=m_{1}$ is shown against $W_{13}=m_{2}$, and component $\Theta_{32}=p_{1}$ is shown against $\Theta_{13}=p_{2}$.}
\label{phase_kir_LSK}
\end{figure}
\subsection{Incompressible 2-D MHD equations}

Here, we demonstrate on low-dimensional matrices with $N=5$ that the integrator \eqref{MHDmethod1} preserves the underlying geometry, namely Casimirs and Hamiltonian\footnote{Numerical simulations for this section are implemented in a Python code available at \url{https://github.com/michaelroop96/qflowMHD.git}}. The final time of the simulation is $T=7500$.

Variations of spectrum of $\Theta$ and $\mathrm{tr}(W\Theta)$ are presented in Fig.~\ref{spec-figure}.
\begin{figure}[h]
\begin{minipage}[h]{0.5\linewidth}
\center{\includegraphics[scale=0.4]{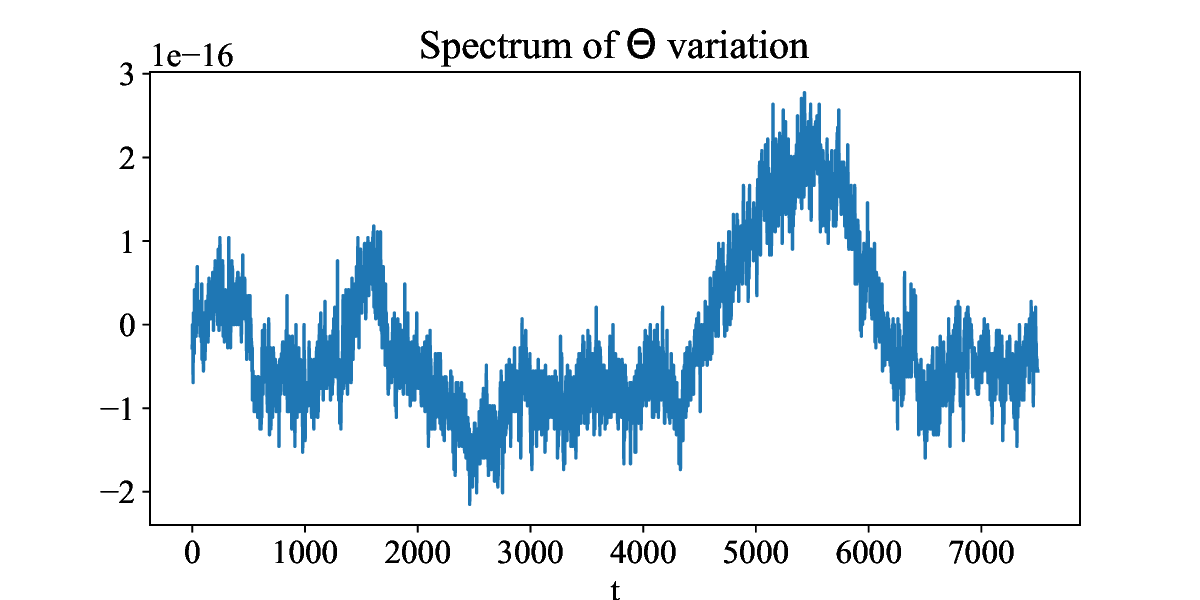}}
\end{minipage}%
\hfill
\begin{minipage}[h]{0.5\linewidth}
\center{\includegraphics[scale=0.4]{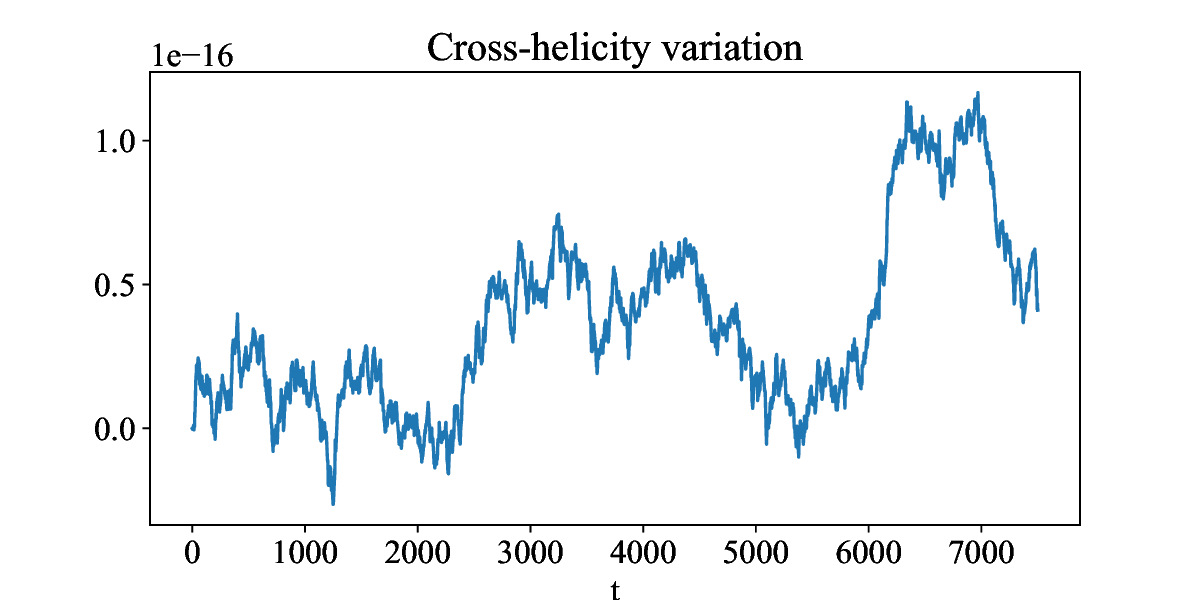}}
\end{minipage}
\hfill
\caption{Variation of the smallest eigenvalue of $\Theta$ and cross-helicity $\mathrm{tr}(W\Theta)$ for incompressible MHD equations. The order $10^{-16}$ of the magnitude of the variation indicates the exact preservation of the Casimirs.}
\label{spec-figure}
\end{figure}

One can see from Fig.~\ref{spec-figure} that variation of Casimirs has the magnitude $10^{-16}$, which is the tolerance of the fixed point iteration that is used to find $(\tilde W,\tilde\Theta)$. This indicates that the Casimirs are exactly preserved.

\begin{figure}[ht!]
\centering
\includegraphics[scale=0.45]{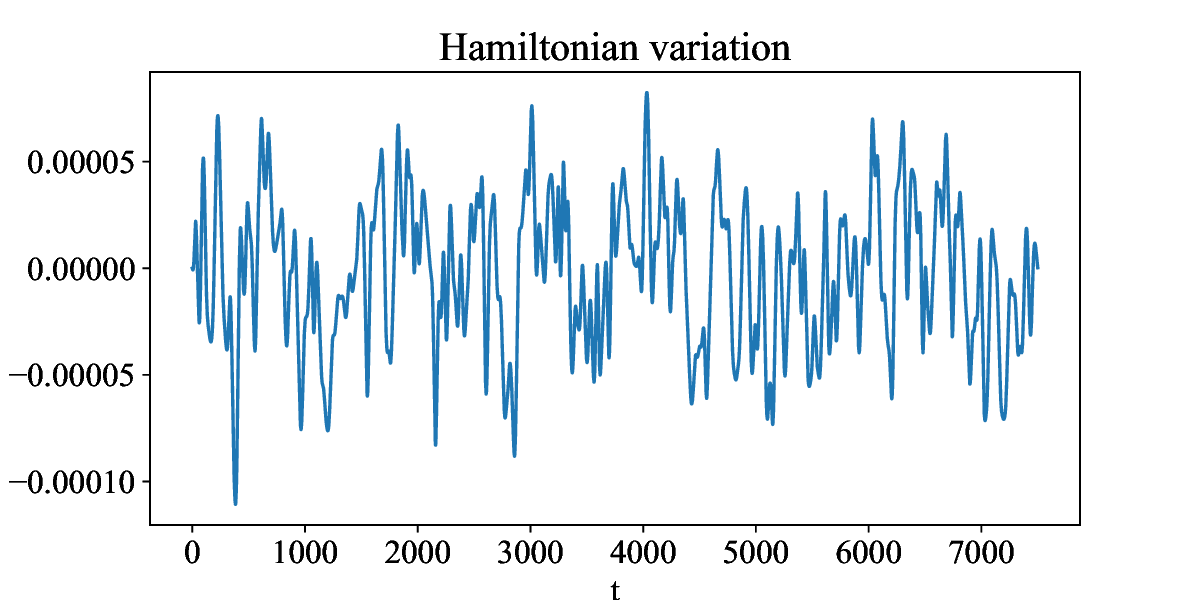}
\caption{Variation of the Hamiltonian for incompressible MHD equations. Absence of drift indicates nearly preservation of the Hamiltonian.}
\label{hamilt-figure}
\end{figure}

In Fig.~\ref{hamilt-figure}, one can see nearly preservation of the Hamiltonian function. The magnitude of the variation is related to the error constant of the method.
\subsection{Hazeltine's equations}
Here\footnote{Numerical simulations for this section are implemented in a Python code available at \url{https://github.com/michaelroop96/qflowAlfven.git}}, we demonstrate the properties stated in Theorem~\ref{thm1} on low-dimensional matrices with $N=5$, $\alpha=2$. 
The final time of simulation is $T=7500$.
The variations in the spectrum of $\Theta$ and $W-\chi$ are presented in Fig.~\ref{spec-figure-alf}.
Variations of cross-helicity $\mathrm{tr}(\chi\Theta)$ and Hamiltonian are presented in Fig.~\ref{crosshel-figure-alf}.
\begin{figure}[h]
\begin{minipage}[h]{0.5\linewidth}
\center{\includegraphics[scale=0.4]{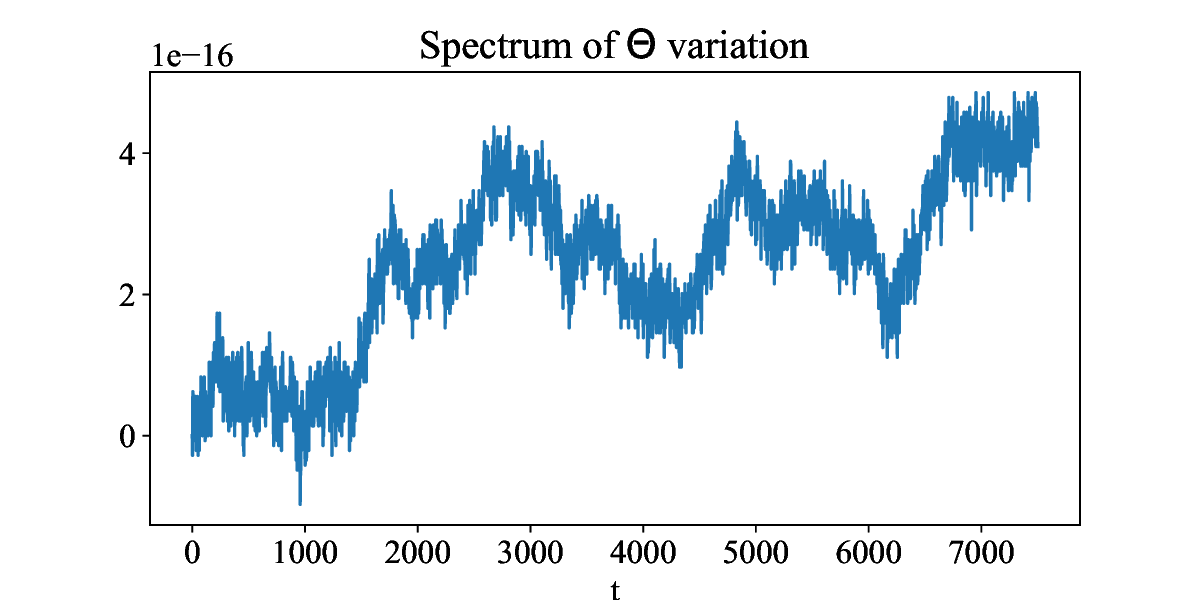}}
\end{minipage}%
\hfill
\begin{minipage}[h]{0.5\linewidth}
\center{\includegraphics[scale=0.4]{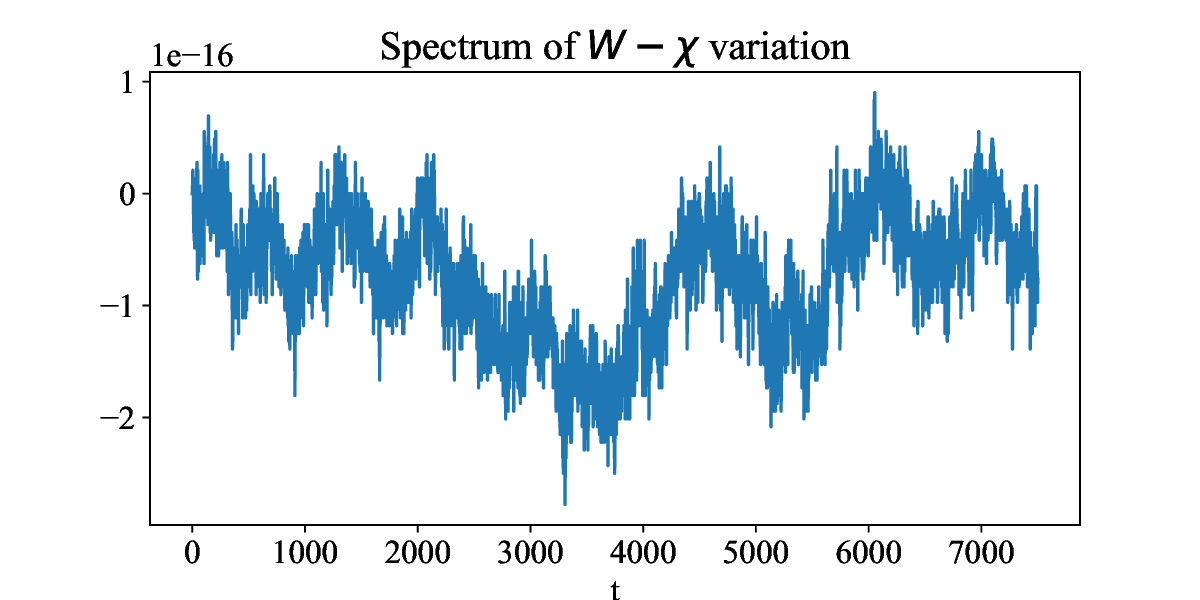}}
\end{minipage}
\hfill
\caption{Variation of the smallest eigenvalue of $\Theta$ and $W-\chi$ for Hazeltine's equations. The order of the magnitude variation indicates the exact preservation of the spectrum.}
\label{spec-figure-alf}
\end{figure}

\begin{figure}[h]
\begin{minipage}[h]{0.5\linewidth}
\center{\includegraphics[scale=0.4]{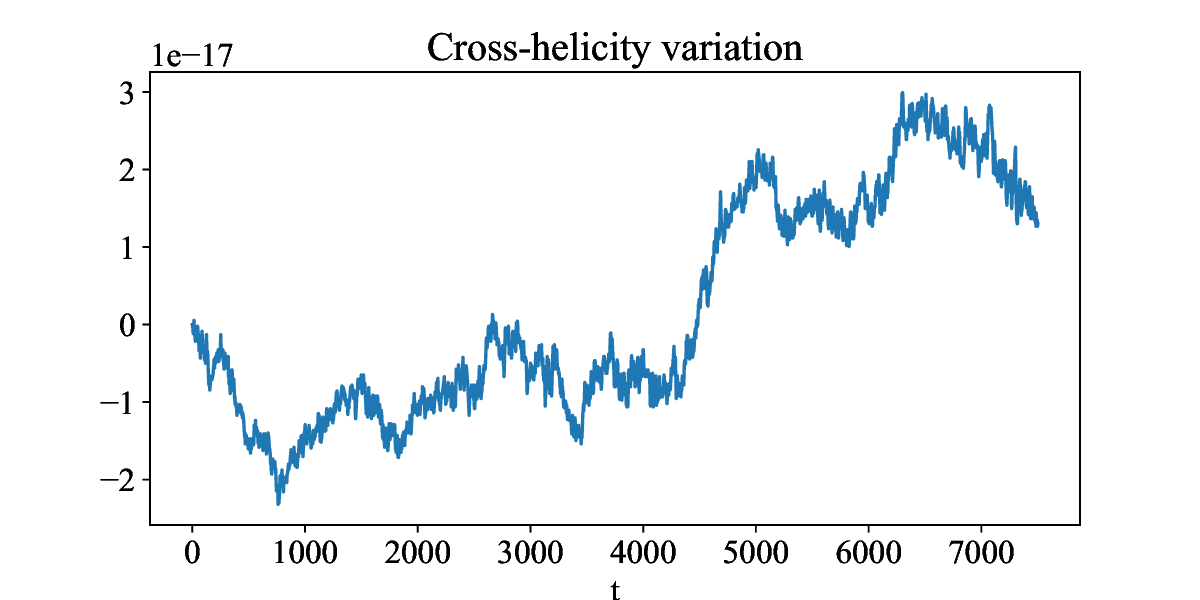}}
\end{minipage}%
\hfill
\begin{minipage}[h]{0.5\linewidth}
\center{\includegraphics[scale=0.4]{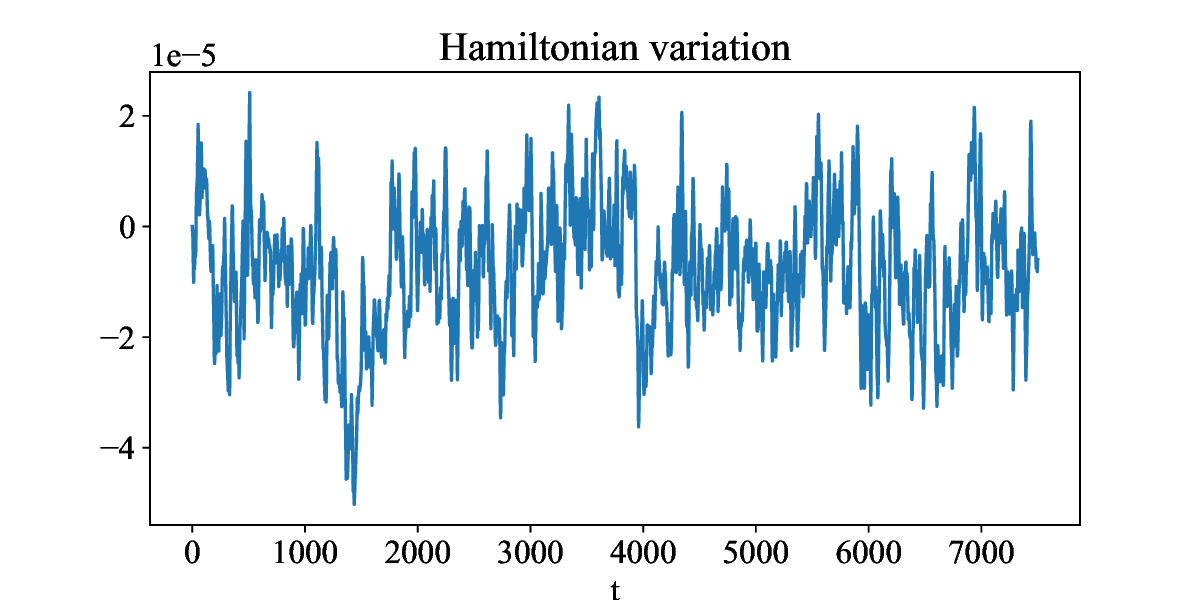}}
\end{minipage}
\hfill
\caption{Variation of the cross-helicity $\mathrm{tr}(\chi\Theta)$ and the Hamiltonian for Hazeltine's equations. The order of the magnitude variation indicates the exact preservation of the cross-helicity.}
\label{crosshel-figure-alf}
\end{figure}

\bibliographystyle{spmpsci}
\bibliography{main.bib}

\begin{thebibliography}{10}
\providecommand{\url}[1]{{#1}}
\providecommand{\urlprefix}{URL }
\expandafter\ifx\csname urlstyle\endcsname\relax
  \providecommand{\doi}[1]{DOI~\discretionary{}{}{}#1}\else
  \providecommand{\doi}{DOI~\discretionary{}{}{}\begingroup \urlstyle{rm}\Url}\fi

\bibitem{Arn}
Arnold, V.: Sur la g\'eometri\'e diff\'erentielle des groupes de lie de dimension infnie et ses applications \'a l'hydrodynamique des fluides parfaits.
\newblock Ann. Inst. Fourier (Grenoble) \textbf{16}, 319--361 (1966)

\bibitem{ArnKh}
Arnold, V., Khesin, B.: Topological Methods in Hydrodynamics.
\newblock Springer Nature Switzerland AG, Cham (2021)

\bibitem{BoMa2016}
Bogfjellmo, G., Marthinsen, H.: High-order symplectic partitioned {L}ie group methods.
\newblock Found. Comp. Math. \textbf{16}(2), 493--530 (2016)

\bibitem{Bord1991}
Bordemann, M., Hoppe, J., Schaller, P., Schlichenmaier, M.: $\mathfrak{gl}(\infty)$ and geometric quantization.
\newblock Commun. Math. Phys. \textbf{138}(2), 209--244 (1991)

\bibitem{Bord1994}
Bordemann, M., Meinrenken, E., Schlichenmaier, M.: Toepliz quantization of {K}\"ahler manifolds and $\mathfrak{gl}(n),\,n\to\infty$ limits.
\newblock Commun. Math. Phys. \textbf{165}(2), 281--296 (1994)

\bibitem{ChanScov}
Channell, P., Scovel, J.: Integrators for {L}ie--{P}oisson dynamical systems.
\newblock Phys. D \textbf{50}, 80--88 (1991)

\bibitem{ChanScov1}
Channell, P., Scovel, J.: Equivariant constrained symplectic integration.
\newblock J. Nonlinear Sci. \textbf{5}, 233--256 (1995)

\bibitem{ChPolt2017}
Charles, L., Polterovich, L.: Sharp correspondence principle and quantum measurements.
\newblock St. Petersburg Math. J. \textbf{29}(1), 177--207 (2018)

\bibitem{CiViMo2023}
Cifani, P., Viviani, M., Modin, K.: An efficient geometric method for incompressible hydrodynamics on the sphere.
\newblock J. Comput. Phys. \textbf{473}, 111772 (2023)

\bibitem{Cleb}
Clebsch, A.: {\"Uber die Bewegung eines K\"orpers in einer Fl\"ussigkeit}.
\newblock Math. Ann. \textbf{3}, 238--262 (1870)

\bibitem{HaiLubWan}
Hairer, E., Lubich, C., Wanner, G.: Geometric Numerical Integration.
\newblock Springer-Verlag Berlin Heidelberg (2006)

\bibitem{Haz}
Hazeltine, R.: Reduced magnetohydrodynamics and the {H}asegawa--{M}ima equation.
\newblock Phys. Fluids \textbf{26}, 3242--3245 (1983)

\bibitem{HazHolm}
Hazeltine, R., Holm, D., Morrison, P.: Electromagnetic solitary waves in magnetized plasmas.
\newblock J. Plasma Phys. \textbf{34}(1), 103--114 (1985)

\bibitem{HazMeiss}
Hazeltine, R., Meiss, J.: Shear-{A}lfv\'en dynamics of toroidally confined plasmas.
\newblock Phys. Rep. \textbf{121}(1-2), 1--164 (1985)

\bibitem{Holm}
Holm, D.: {H}amiltonian structure for {A}lfv\'{e}n wave turbulence equations.
\newblock Phys. Lett. A \textbf{108A}, 445--447 (1985)

\bibitem{HMR}
Holm, D., Marsden, J., Ratiu, T.: The {E}uler--{P}oincar\'{e} equations and semidirect products with applications to continuum theories.
\newblock Adv. Math. \textbf{137}, 1--81 (1998)

\bibitem{Hopp}
Hoppe, J.: Quantum theory of a massless relativistic surface and a two-dimensional bound state problem. PhD thesis.
\newblock MIT (1982)

\bibitem{Hopp1}
Hoppe, J.: Diffeomorphism groups, quantization, and $\mathrm{SU}(\infty)$.
\newblock Int. J. Mod. Phys. A \textbf{4}(19) (1989).
\newblock \doi{10.1142/S0217751X89002235}

\bibitem{HoppYau}
Hoppe, J., Yau, S.T.: Some properties of matrix harmonics on $\mathrm{S}^{2}$.
\newblock Comm. Math. Phys. \textbf{195}, 67--77 (1998)

\bibitem{Jay}
Jay, L.: Symplectic partitioned {R}unge-{K}utta methods for constrained hamiltonian systems.
\newblock SIAM J. Numer. Anal. \textbf{33}(1), 368--387 (1996)

\bibitem{KhMisMod}
Khesin, B., Misiolek, G., Modin, K.: Geometric hydrodynamics and infinite-dimensional {N}ewton's equations.
\newblock Bull. Amer. Math. Soc. \textbf{58}, 377--442 (2021)

\bibitem{KhPerYang2019}
Khesin, B., Peralta-Salas, D., Yang, C.: A basis of {C}asimirs in {3D} magnetohydrodynamics.
\newblock Int. Math. Res. Not. \textbf{2021}(18), 13645--13660 (2020)

\bibitem{Kir}
Kirchhoff, G.: Vorlesungen \"{u}ber mathematische Physik. Mechanik.
\newblock Leipzig, Teubner (1876)

\bibitem{KrTassGrass2016}
Kraus, M., Tassi, E., Grasso, D.: Variational integrators for reduced magnetohydrodynamics.
\newblock J. Comput. Phys. \textbf{321}, 435--458 (2016)

\bibitem{MarsPekSh}
Marsden, J., Pekarsky, S., Shkoller, S.: Discrete {E}uler--{P}oincar\'{e} and {L}ie--{P}oisson equations.
\newblock Nonlinearity \textbf{12}, 1647--1662 (1999)

\bibitem{MarsRat}
Marsden, J., Ratiu, T.: Introduction to Mechanics and Symmetry.
\newblock Springer-Verlag, New-York (1999)

\bibitem{MRW}
Marsden, J., Ratiu, T., Weinstein, A.: Reduction and {H}amiltonian structures on duals of semidirect product {L}ie algebras.
\newblock Contemp. Math. \textbf{28}, 55--100 (1984)

\bibitem{MarsRatWein}
Marsden, J., Ratiu, T., Weinstein, A.: Semidirect products and reduction in mechanics.
\newblock Trans. Am. Math. Soc. \textbf{281}(1), 147--177 (1984)

\bibitem{McLachModVerd}
McLachlan, R., Modin, K., Verdier, O.: Collective symplectic integrators.
\newblock Nonlinearity \textbf{27}(6), 1525--1542 (2014)

\bibitem{McLModVerdWil}
McLachlan, R., Modin, K., Verdier, O., Wilkins, M.: Geometric generalisations of {SHAKE} and {RATTLE}.
\newblock Found. Comput. Math. \textbf{14}(2), 339--370 (2014)

\bibitem{McQui1}
McLachlan, R., Quispel, G.: Splitting methods.
\newblock Acta Numer. \textbf{11}, 341--434 (2002)

\bibitem{McQui2}
McLachlan, R., Quispel, G.: Explicit geometric integration of polynomial vector fields.
\newblock BIT Num. Math. \textbf{44}, 515--538 (2004)

\bibitem{MoPe2024}
Modin, K., Perrot, M.: {E}ulerian and {L}agrangian stability in {Z}eitlin's model of hydrodynamics.
\newblock Comm. Math. Phys. \textbf{405}, 177 (2024)

\bibitem{ModViv1}
Modin, K., Viviani, M.: A {C}asimir preserving scheme for long-time simulation of spherical ideal hydrodynamics.
\newblock J. Fluid Mech. \textbf{884}, A22 (2020)

\bibitem{ModViv}
Modin, K., Viviani, M.: {L}ie-{P}oisson methods for isospectral flows.
\newblock Found. Comput. Math. \textbf{20}, 889--921 (2020)

\bibitem{Mo2017a}
Morrison, P.: Structure and structure-preserving algorithms for plasma physics.
\newblock Phys. Plasmas \textbf{24}(5), 055502 (2017)

\bibitem{MoGr1980}
Morrison, P., Greene, J.: Noncanonical {H}amiltonian density formulation of hydrodynamics and ideal magnetohydrodynamics.
\newblock Phys. Rev. Lett. \textbf{45}(10), 790--794 (1980)

\bibitem{Sa1988}
Sanz-Serna, J.: Runge-{K}utta schemes for {H}amiltonian systems.
\newblock BIT Numerical Mathematics \textbf{28}(4), 877--883 (1988)

\bibitem{ShukYuRah}
Shukla, P., Yu, M., Rahman, H., Spatschek, K.: Nonlinear convective motion in plasmas.
\newblock Phys. Rep. \textbf{105}(4-5), 227--328 (1984)

\bibitem{LSK}
Steklov, V.: {\"Uber die Bewegung eines festen K\"orpers in einer Fl\"ussigkeit}.
\newblock Math. Ann. \textbf{42}, 273--274 (1893)

\bibitem{ThMo2000}
Thiffeault, J.L., Morrison, P.: Classification and {C}asimir invariants of {L}ie--{P}oisson brackets.
\newblock Phys. D \textbf{136}(3-4), 205--244 (2000)

\bibitem{VD}
Vishik, S., Dolzhanski, F.: Analogs of the {E}uler-{L}agrange equations and magnetohydrodynamis equations related to {L}ie groups.
\newblock Sov. Math. Doklady \textbf{19}, 149--153 (1978)

\bibitem{Milo}
Viviani, M.: A minimal-variable symplectic method for isospectral flows.
\newblock BIT Num. Math. \textbf{60}, 741--758 (2020)

\bibitem{Ze1991}
Zeitlin, V.: Finite-mode analogs of {$2$}{D} ideal hydrodynamics: coadjoint orbits and local canonical structure.
\newblock Phys. D \textbf{49}(3), 353--362 (1991)

\bibitem{Zeit}
Zeitlin, V.: Self-consistent-mode approximation for the hydrodynamics of an incompressible fluid on non rotating and rotating spheres.
\newblock Phys. Rev. Lett. \textbf{93}(26), 264501 (2004)

\bibitem{Ze2005}
Zeitlin, V.: On self-consistent finite-mode approximations in (quasi-) two-dimensional hydrodynamics and magnetohydrodynamics.
\newblock Phys. Lett. A \textbf{339}(3-5), 316--324 (2005)

\end{thebibliography}
\end{document}